\documentclass[11pt]{article}
\usepackage{amsmath}
\usepackage{bbm}
\usepackage{amsmath,amsthm,amssymb,amscd}
\usepackage{latexsym}
\usepackage{hyperref}
\usepackage{indentfirst}
\usepackage[numbers,sort&compress]{natbib}

\newtheorem{theorem}{Theorem}[section]

\newtheorem{lemma}{Lemma}[section]
\newtheorem{definition}{Definition}[section]
\newtheorem{remark}{Remark}[section]

\numberwithin{equation}{section} 

\title{{\Large \bf Monotonicity formula and Liouville-type theorems of stable solution for the weighted elliptic system}\thanks{The work was partially supported by NSFC of China (No. 11201248), K.C. Wong Fund of Ningbo University and
Ningbo Natural Science Foundation (No. 2014A610027).}}
\author{{\small Liang-Gen Hu\footnote{email address: hulianggen@tom.com}}\\[0.05cm] {\small Department of Mathematics,
Ningbo University, 315211, P.R. China}}
\date{}
\begin{document}
\maketitle
\def\abstractname{}
\vspace*{-35pt}

\begin{abstract}
\noindent {\bf Abstract:} In this paper, we are concerned with the weighted elliptic system
\begin{equation*}
\begin{cases}
-\Delta u=|x|^{\beta} v^{\vartheta},\\
-\Delta v=|x|^{\alpha} |u|^{p-1}u,
\end{cases}\quad \mbox{in}\;\ \Omega,
\end{equation*}where $\Omega$ is a subset of $\mathbb{R}^N$, $N \ge 5$, $\alpha >-4$, $0 \le \beta \le \dfrac{N-4}{2}$, $p>1$ and $\vartheta=1$. We first apply Pohozaev identity to construct a monotonicity formula
and reveal their certain equivalence relation. By the use of
{\it Pohozaev identity}, {\it monotonicity formula} of solutions together with a {\it blowing down} sequence,
we prove Liouville-type theorems of stable solutions (whether positive or sign-changing) for the weighted elliptic system
in the higher dimension.
\\ [0.2cm]
\end{abstract}\vskip .2in

\section{Introduction}

In this article, we examine the nonexistence of classical stable solutions of the weighted elliptic system given by
\begin{equation}\label{eq:1.1}
\begin{cases}
-\Delta u=|x|^{\beta} v^{\vartheta},\\
-\Delta v=|x|^{\alpha} |u|^{p-1}u,
\end{cases}\quad \mbox{in}\;\ \Omega,
\end{equation}where $\Omega$ is a subset of $\mathbb{R}^N$, $N\ge 5$, $\alpha>-4$, $0 \le \beta \le \dfrac{N-4}{2}$ and $p \theta>1$.\vskip .06in

The idea of using the Morse index of a solution for a semilinear elliptic equation was first explored by Bahri and Lions \cite{Bahri}
to get further qualitative properties of the solution. Recently, along this line of research, Dancer
\cite{Dancer-TAMS,Dancer-IUMJ,Dancer-JRAM} introduced the finite Morse index solution
and made the significant progress in the elliptic equations. Let us note that the solution $u$ is stable if and only if its Morse index is equal to zero.
In 2007, Farina considered the Lane-Emden equation
\begin{equation}\label{eq:1.2}
- \Delta u =|u|^{p-1}u,
\end{equation}on bounded and unbounded domains of $\Omega \subset \mathbb{R}^N$, with $N \ge 2$ and $p >1$. Based on a delicate application of
the classical Moser's iteration, he gave the complete
classification of finite Morse index solutions (positive or sign-changing) in his seminal paper \cite{Farina}. Hereafter, many experts utilized the Moser's iterative method to discuss the stable and finite Morse index solutions of the harmonic and fourth-order
elliptic equation and obtained many excellent results. We refer to \cite{Dancer-1,Wang-Ye,Wei-Ye,Wei-Xu-Yang}
and the reference therein.\vskip .06in

However, the classical Moser's iterative technique does not completely
classify finite Morse index solutions of the biharmonic equation
\begin{equation*}
\Delta ^2 u=|u|^{p-1}u,\quad \mbox{in} \;\ \Omega \subset \mathbb{R}^N.
\end{equation*}
To solve the problem, D\'{a}vila et al. \cite{Davila} have recently
derived a monotonicity formula of solutions and given the complete classification of stable and finite Morse index solutions for the biharmonic equation by the application of Pohozaev identity and the monotonicity formula. We note that many outstanding papers \cite{Davila-JFA,Davila,Pacard-1,Pacard-2,Wang} utilize a monotonicity formula to study the partial regularity of stationary weak solution, stable and finite Morse index solutions for the harmonic and fourth-order equation.\vskip .06in

On the other hand, some experts were interesting in the Lane-Emden system and obtained some excellent results \cite{Cowan-Nonlinearity,Esposito-PRSE,Fazly-ANS,Fazly-Ghoussoub}. In 2013, applying a iterative method and the pointwise estimate in \cite{Souplet}, Cowan proved the following result.\vskip .06in

\noindent {\bf Theorem A.} (\cite[Theorem 2]{Cowan-Nonlinearity}) {\it Suppose that $p>\theta=1$, $\alpha=\beta=0$ and
\begin{equation*}
N<2+\dfrac{4(p+1)}{p-1}\left (\sqrt{\dfrac{2p}{p+1}}+\sqrt{\dfrac{2p}{p+1}-\sqrt{\dfrac{2p}{p+1}}}\right ).
\end{equation*}Then there is no positive stable solution of (\ref{eq:1.1}).}\vskip .06in

Adopting the same method as Cowan \cite{Cowan-Nonlinearity}, Fazly obtained the following result.\vskip .06in

\noindent {\bf Theorem B.} (\cite[Theorem 2.4]{Fazly-ANS}) {\it Suppose that $(u,v)$ is $C^2(\mathbb{R}^N)$ nonnegative entire semi-stable solution of
\begin{equation*}
\begin{cases}
-\Delta u=\rho (1+|x|^2)^{\frac{\alpha}{2}}v,\\
-\Delta v=\varrho (1+|x|^2)^{\frac{\alpha}{2}}u^p,
\end{cases}
\end{equation*}
with $\rho, \varrho >0$ in the dimension
\begin{equation*}
N<8+3\alpha+\dfrac{8+4\alpha}{p-1}.
\end{equation*}Then, $(u,v)$ is the trivial solution.}\vskip .06in

We observe that the dimension $N<8+3\alpha+\dfrac{8+4\alpha}{p-1}$ in \cite[Theorem 2.4]{Fazly-ANS}
is already larger than the {\it critical hyperbola}, i.e., $N=4+\alpha+\dfrac{8+4\alpha}{p-1}$.
Recently, Fazly and Ghoussoub
\cite[Theorem 4]{Fazly-Ghoussoub} have considered the nonexistence of positive stable solutions for the weighted
elliptic system (\ref{eq:1.1}), which the dimension satisfies
\begin{equation*}
N<2+2\left (\dfrac{p(\beta+2)+\alpha+2}{p \theta-1} \right ) \left (\sqrt{\dfrac{p\theta (\theta+1)}{p+1}}+\sqrt{\dfrac{p \theta (\theta+1)}{p+1}-\sqrt{\dfrac{p\theta (\theta+1)}{p+1}}}\right ).
\end{equation*}Clearly, if $\theta=1$ and $\alpha=\beta=0$ in (\ref{eq:1.1}), then their result is the same as Theorem A.\vskip .1in

Let us briefly recall the fact that Liouvile-type theorem of solutions for various Lane-Emden equations and systems is interesting and challenging for decades.\vskip .06in

First, Pohozaev identity shows that the Lane-Emden equation with the Dirichlet boundary condition has no positive
solution on a bounded star-shaped domain $\Omega \subset \mathbb{R}^N$, whenever $p \ge \dfrac{N+2}{N-2}$.  On the other hand, Gidas and Spruck obtained the optimal Liouville-type theorems in the celebrated paper \cite{Gidas}, that is, the Lane-Emden equation (\ref{eq:1.2}) has no positive solution if and only if $1<p <\dfrac{N+2}{N-2}(=+\infty,$ if $N\le 2)$. In 1991, Bidaut-V\'{e}ron and V\'{e}ron \cite{Bidaut} obtained the asymptotic behavior of positive solution by utilizing the Bochner-Lichnerowicz-Weitzenb\"{o}ck formula in $\mathbb{R}^N$.\vskip .06in

In the case of the Lane-Emden systems (\ref{eq:1.1}) with $\alpha=\beta=0$, Pucci and Serrin \cite{Pucci} proved that if $\dfrac{N}{p+1}+\dfrac{N}{\theta+1} \le N-2$ and $\Omega$ is a bounded star-shaped domain of $\mathbb{R}^N$, then there is no positive solution of (\ref{eq:1.1}) with the Dirichlet boundary conditions. Noting that the curve $\dfrac{N}{p+1}+\dfrac{N}{\theta+1}=N-2$ is the {\it critical Sobolev hyperbola}. Similar to the Lane-Emden equation, the following conjecture is interesting and challenging.\vskip .06in

\noindent {\bf Conjecture} ({\it Lane-Emden Conjecture}) {\it Suppose $(p,\theta)$ is under the critical Sobolev hyperbola, i.e.,
\begin{equation*}
\dfrac{N}{p+1}+\dfrac{N}{\theta+1}>N-2.
\end{equation*}Then there is no positive solution for the elliptic system (\ref{eq:1.1}) with $\alpha=\beta=0$.}\vskip .06in

The case of radial solutions was solved by Mitidieri \cite{Mitidieri} in any dimension, and the positive radial solutions on and above the critical Sobolev hyperbola was constructed by \cite{Mitidieri,Serrin-Zou-98}, which is the optimal Liouville-type theorem for radial solutions. The {\it conjecture} (for non-radial solutions) seems difficult. In the dimension $N=3$, Serrin and Zou \cite{Serrin-Zou-96} proved the {\it conjecture} for the polynomially bounded solutions, which the boundedness was removed in \cite{Polacik}.
In 2009, Souplet \cite{Souplet} solved the {\it conjecture} in $N=4$ or a new
region for $N \ge 5$.
However, the weighted Lane-Emden system (\ref{eq:1.1}) is even less understood. For example,
the paper \cite{Phan} proved the {\it conjecture}
for the equation $-\Delta u=|x|^{\alpha} u^p$ in $N=3$; In 2012, Phan \cite{Phan-ADE}
solved the {\it conjecture} for the system (\ref{eq:1.1}) in two cases: {\it case 1}. $N=3$ and bounded solutions; {\it case 2}. $N=3$ or $4$ and $\alpha, \beta \le 0$.\vskip .1in

Here and in the following, we always assume that $N\ge 5$, $\alpha>-4$, $0\le \beta \le \dfrac{N-4}{2}$, $p>1$ and $\theta=1$.
Motivated by the ideas in \cite{Davila,Du,Hu}, we will construct a monotonicity formula of solutions in the dimension $4+\beta+\dfrac{8+2\alpha+2\beta}{p-1}<N<N_{\alpha,\beta}(p)$ ($N_{\alpha,\beta}(p)$ see below (\ref{eq:3.1})) and get various integral estimates, and then use these results to study Liouville-type theorems of stable solution for the weighted elliptic system (\ref{eq:1.1}).
\vskip .1in

\begin{theorem}\label{eq:t1.1}
For any $4+\beta+\dfrac{8+2\alpha+2\beta}{p-1}<N<N_{\alpha,\beta}(p)$, assume that $u\in
W^{2,2}_{loc}(\mathbb{R}^N\backslash \{0\})$ is a homogeneous, stable solution of (\ref{eq:1.1}), $|x|^{\alpha}|u|^{p+1} \in L_{loc}^1(\mathbb{R}^N
\backslash \{0\})$ and $|x|^{-\beta}|\Delta u|^2 \in L_{loc}^1(\mathbb{R}^N\backslash \{0\})$. Then
$u \equiv 0$.
\end{theorem}

Applying Theorem \ref{eq:t1.1} and the properties of monotonicity formula (\ref{eq:2.14}), we get

\begin{theorem}\label{eq:t1.2}
If $u\in C^4(\mathbb{R}^N)$ is a stable solution of (\ref{eq:1.1}) in $\mathbb{R}^N$ and $5 \le N \le N_{\alpha,\beta}(p)$, then $u \equiv 0$.
\end{theorem}

\begin{remark}
\begin{itemize}
\item [\rm (1)] We apply Pohozaev identity to construct a monotonicity formula. From the process of the proof in Theorem \ref{eq:t2.1}, we can
observe that Pohozaev identity is equivalence to the certain derivative-type of the monotonicity formula.
\item [\rm (2)] Let us note that for the dimensions $4+\beta+\dfrac{8+2\alpha+2\beta}{p-1}<N<N_{\alpha,\beta}(p)$, we adopt a new
method of monotonicity formula together with blowing down sequence to investigate Liouville-type theorem. In addition,
a difficulty stems from the fact that the terms $|x|^{\alpha}$ and $|x|^{\beta}$ in (\ref{eq:1.1}) leads to the singularity. For this reason, we use a more delicate approach to derive
improved integral estimates.
\item[\rm (3)]  From the computation of $N_{\alpha,\beta}(p)$ (in Section 3), we find the following relation:
\begin{equation*}
\begin{cases}
N_{0,0}(p)>2+\dfrac{4(p+1)}{p-1}\left (\sqrt{\dfrac{2p}{p+1}}+\sqrt{\dfrac{2p}{p+1}-\sqrt{\dfrac{2p}{p+1}}}\right ),& \mbox{if}\;\ \alpha =\beta =0,\\
N_{\alpha,\alpha}(p)>8+3\alpha+\dfrac{8+4\alpha}{p-1},& \mbox{if}\;\  \alpha=\beta.
\end{cases}
\end{equation*}
Therefore, in contrast with {\bf Theorem A} and {\bf Theorem B}, we obtain Liouville-type theorem in the higher dimension.
\end{itemize}
\end{remark}

Next, we list some definitions and notations. Let $\Omega$ be a subset of $\mathbb{R}^N$ and $f,g \in C^1\left (\mathbb{R}^{N+2},\Omega \right )$. Following Montenegro \cite{Montenegro}, we consider the general elliptic system
\begin{equation*}
(S_{f,g})\;
\begin{cases}
-\Delta u=f(u,v,x),\\
-\Delta v=g(u,v,x),
\end{cases}\quad x \in \Omega.
\end{equation*}A solution $(u,v)\in C^2(\Omega)\times C^2(\Omega)$ of $(S_{f,g})$ is called {\it stable}, if the eigenvalue problem
\begin{equation*}
(E_{f,g})
\begin{cases}
-\Delta \phi=f_u(u,v,x)\phi+f_v(u,v,x)\psi+\eta \phi,\\
-\Delta \psi =g_u(u,v,x)\phi+g_v(u,v,x)\psi+\eta \psi,
\end{cases}
\end{equation*} has a first positive eigenvalue $\eta>0$, with corresponding positive smooth eigenvalue pair $(\phi,\psi)$.
A solution $(u,v)$ is said to be {\it semi-stable}, if the first eigenvalue $\eta$ is nonnegative.\vskip .06in

Inspired by the above definition, we give the integration-type definition of stability.

\begin{definition}
We recall that a critical point $u \in C^4(\Omega)$ of the energy function
\begin{equation*}
\mathcal{E}(u)=\int_{\Omega} \left [\dfrac{1}{2}\dfrac{|\Delta u|^2}{|x|^{\beta}} -\dfrac{1}{p+1} |x|^{\alpha} |u|^{p+1} \right ]dx
\end{equation*}is said to be a stable solution of (\ref{eq:1.1}), if, for any $\zeta \in C_0^2 (\Omega)$, we have
\begin{equation*}
p\int_{\Omega} |x|^{\alpha}|u|^{p-1}\zeta^2 dx \le \int_{\Omega} \dfrac{|\Delta \zeta|^2}{|x|^{\beta}} dx.
\end{equation*}
\end{definition}

The definition is interesting and well-defined. In deed, if $(u,v)$ is a semi-stable solution, then there exist $\eta \ge 0$ and a positive smooth eigenvalue pair $(\phi,\psi)$ such that
\begin{equation*}
\begin{cases}
-\Delta \phi = |x|^{\beta} \psi+\eta \phi,\\
-\Delta \psi = p|x|^{\alpha}|u|^{p-1}\phi+\eta \psi.
\end{cases}
\end{equation*}Multiply the second equation by $\dfrac{\zeta^2}{\phi}$ with $\zeta \in C_0^2(\Omega)$ to get
\begin{align}\label{eq:1.3}
& p\int_{\Omega} |x|^{\alpha}|u|^{p-1}\zeta^2 dx \le \int_{\Omega} -\Delta \psi \dfrac{\zeta^2}{\phi} dx
=\int_{\Omega}-\psi\Delta \left (\dfrac{\zeta^2}{\phi} \right ) dx \nonumber \\[0.1cm]
& = \int_{\Omega}\dfrac{1}{|x|^{\beta}}\left [1-\dfrac{\eta \phi}{|x|^{\beta}\psi+\eta \phi} \right ] \Delta \phi\Delta
\left (\dfrac{\zeta^2}{\phi}\right )dx.
\end{align}A simple calculation leads to
\begin{equation*}
\Delta \left (\frac{\zeta^2}{\phi}\right )=2\phi^{-1}|\nabla \zeta|^2+2\zeta\phi^{-1}\Delta \zeta-4\zeta\phi^{-2}\nabla \zeta\cdot \nabla \phi
+2\zeta^2\phi^{-3} |\nabla \phi|^2-\zeta^2\phi^{-2}\Delta \phi.
\end{equation*}Then we find
\begin{align*}
\Delta \phi \Delta \left (\dfrac{\zeta^2}{\phi} \right ) & -|\Delta \zeta|^2=2\zeta \phi^{-1}\Delta \zeta \Delta \phi
-\zeta^2\phi^{-2}|\Delta \phi|^2 -|\Delta \zeta|^2 \\
& +2 \phi^{-1}\Delta \phi [ |\nabla \zeta|^2-2\zeta\phi^{-1} \nabla \zeta \cdot \nabla \phi +\zeta^2 \phi^{-2} |\nabla \phi|^2 ] \\
= &- \left [(\zeta \phi^{-1}\Delta \phi -\Delta \zeta)^2 +2(\phi^{-1} |x|^{\beta}\psi+\eta) (\nabla \zeta- \zeta \phi^{-1} \nabla \phi)^2\right ] \\
\le & 0,
\end{align*}implies
\begin{equation*}
\int_{\Omega} \dfrac{\Delta \phi}{|x|^{\beta}} \Delta \left (\dfrac{\zeta^2}{\phi} \right ) dx
\le \int_{\Omega} \dfrac{|\Delta \zeta|^2}{|x|^{\beta}} dx.
\end{equation*}Therefore, combining the above inequality with (\ref{eq:1.3}), we obtain
\begin{equation*}
p\int_{\Omega} |x|^{\alpha} |u|^{p-1}\zeta^2 dx \le \int_{\Omega} \dfrac{|\Delta \zeta|^2}{|x|^{\beta}} dx.
\end{equation*}

\begin{remark}
Since $\phi$ is a smooth function, $\zeta \in C_0^2(\Omega)$ and $\beta \le \dfrac{N-4}{2}$, then the integration
$\displaystyle \int_{\Omega} \dfrac{1}{|x|^{\beta}}dx $ is well defined.
\end{remark}\vskip .06in

\noindent {\bf Notations.} Throughout this paper, $B_r(x)$ denotes the open ball of radius $r$ centered at $x$.
If $x=0$, we simply denote $B_r(0)$ by $B_r$. $C$ denotes various irrelevant
positive constants.\vskip .06in

The rest of the paper is organized as follows. In Section 2, we derive various
integral estimates and construct a monotonicity formula. In Section 3, we prove Liouville-type theorem of homogeneous, stable solutions in the dimensions $4+\beta+\dfrac{8+2\alpha+2\beta}{p-1}<N< N_{\alpha,\beta}(p)$. Finally, we study the qualitative properties of the monotonicity function
$\mathcal{M}$, and prove Theorem \ref{eq:t1.2} which is based on {\it Pohozaev-type identity}, {\it monotonicity formula} together with {\it blowing down} sequences in Section 4.
\vskip .2in

\section{Some estimates and a monotonicity formula}

\vskip .1in

\begin{lemma}(\cite[Lemma 2.2]{Wei-Ye})\label{eq:l2.1}
For any $\zeta \in C^4(\mathbb{R}^N)$ and $\eta \in C^4(\mathbb{R}^N)$, the identity holds
\begin{equation*}
\Delta \zeta \Delta \left (\zeta\eta^2 \right )=[\Delta (\zeta \eta)]^2-4 (\nabla \zeta \cdot \nabla \eta)^2-\zeta^2 |\Delta \eta|^2
+2\zeta \Delta \zeta |\nabla \eta|^2-4\zeta \Delta \eta \nabla \zeta \cdot \nabla \eta.
\end{equation*}
\end{lemma}

\begin{lemma}\label{eq:l2.2}
For any $\zeta \in C^4(\mathbb{R}^N)$ and $\eta \in C_0^4 (\mathbb{R}^N)$, then the following equalities hold
\begin{align}\label{eq:2.1}
\int_{\mathbb{R}^N} \Delta \left (\dfrac{\Delta \zeta}{|x|^{\beta}} \right ) & \zeta \eta^2 dx = \int_{\mathbb{R}^N} \dfrac{[\Delta(\zeta \eta)]^2}
{|x|^{\beta}}+\int_{\mathbb{R}^N}\dfrac{1}{|x|^{\beta}}\Big [-4(\nabla \zeta \cdot \nabla \eta)^2+2\zeta\Delta \zeta |\nabla \eta|^2 \Big ] dx \nonumber\\[0.2cm]
& +\int_{\mathbb{R}^N} \dfrac{\zeta^2}{|x|^{\beta}}\Big [ 2\nabla (\Delta \eta) \cdot \nabla \eta+|\Delta \eta|^2
-2\beta|x|^{-2}\Delta \eta(x\cdot \nabla \eta) \Big ] dx,
\end{align}and
\begin{align}\label{eq:2.2}
2\int_{\mathbb{R}^N} & \dfrac{|\nabla \zeta|^2|\nabla \eta|^2}{|x|^{\beta}} dx=\int_{\mathbb{R}^N}\left [ \dfrac{2}{|x|^{\beta}}
\zeta (-\Delta \zeta)|\nabla \eta|^2+\dfrac{\zeta^2}{|x|^{\beta}}\Delta
\left (|\nabla \eta|^2\right ) \right ]dx \nonumber \\[0.2cm]
&+\int_{\mathbb{R}^N} \dfrac{\zeta^2}{|x|^{\beta+2}}\Big [\beta(\beta+2-N) |\nabla \eta|^2-2\beta \left (x \cdot \nabla \left (|\nabla \eta|^2 \right )\right ) \Big ]dx.
\end{align}
\end{lemma}

\begin{proof}By the divergence theorem and integration by parts, we get
\begin{align*}
-4\int_{\mathbb{R}^N} & \dfrac{1}{|x|^{\beta}} \zeta \Delta \eta \nabla \zeta \cdot \nabla \eta dx =-2
\int_{\mathbb{R}^N} \dfrac{1}{|x|^{\beta}} \Delta \eta \nabla \zeta^2 \cdot \nabla \eta dx\\[0.18cm]
& = 2\int_{\mathbb{R}^N} \dfrac{\zeta^2}{|x|^{\beta}} \Big [\nabla (\Delta \eta)\cdot \nabla \eta +|\Delta \eta|^2 -\beta |x|^{-2}
\Delta \eta (x \cdot \nabla \eta)\Big ] dx.
\end{align*}Combining with Lemma \ref{eq:l2.1}, it implies that the identity (\ref{eq:2.1}) holds.

On the other hand, it is easy to see that
\begin{equation*}
\dfrac{1}{2}\Delta (\zeta^2)=\zeta \Delta \zeta +|\nabla \zeta|^2,
\end{equation*}then we obtain
\begin{equation*}
\int_{\mathbb{R}^N}\dfrac{|\nabla \zeta|^2|\nabla \eta|^2}{|x|^{\beta}}= \int_{\mathbb{R}^N} \dfrac{\zeta (-\Delta \zeta)|\nabla \eta|^2}{|x|^{\beta}}+\dfrac{1}{2}\int_{\mathbb{R}^N}
\zeta^2 \Delta \left ( \dfrac{ |\nabla \eta|^2}{|x|^{\beta}} \right )dx.
\end{equation*}A direct computation yields
\begin{equation*}
\Delta \left ( \dfrac{ |\nabla \eta|^2}{|x|^{\beta}} \right ) =
\dfrac{1}{|x|^{\beta}} \Big [\beta (\beta+2-N)|x|^{-2}|\nabla \eta|^2-2\beta |x|^{-2}(x\cdot \nabla (|\nabla \eta|^2))+\Delta (|\nabla \eta|^2) \Big ].
\end{equation*}Substituting into the above identity, we get the identity (\ref{eq:2.2}).
\end{proof}

\begin{lemma}\label{eq:al2.3}
Let $u \in C^4(\mathbb{R}^N)$ be a stable solution of (\ref{eq:1.1}). Then we find
\begin{align}\label{eq:2.3}
& \int_{B_R(x)}\left (\dfrac{|\Delta u|^2}{|z|^{\beta}} +|z|^{\alpha}|u|^{p+1} \right )dz \nonumber \\[0.1cm]
& \le CR^{-2} \int_{B_{2R}(x)\backslash B_R(x)}\dfrac{|u\Delta u|}{|z|^{\beta}} dz +CR^{-4}\int_{B_{2R}(x)\backslash B_R(x)}\dfrac{u^2}{|z|^{\beta}} dz.
\end{align}Furthermore, for large enough $m$, we obtain that for any $\psi \in C_0^4(\mathbb{R}^N)$ with $0\le \psi \le 1$
\begin{align*}
 \int_{\mathbb{R}^N}\left [\dfrac{|\Delta u|^2}{|x|^{\beta}}+|x|^{\alpha}|u|^{p+1}\right ] & \psi^{2m} dx
\le C \int_{\mathbb{R}^N}|x|^{-\frac{2\alpha+\beta p+\beta}{p-1}}\mathfrak{Q}(\psi^m)^{\frac{p+1}{p-1}}dx  \\[0.15cm]
& +C \int_{\mathbb{R}^N}|x|^{-\frac{2\alpha+(\beta+2)(p+1)}{p-1}}\mathfrak{R}(\psi^m)^{\frac{p+1}{p-1}}dx,
\end{align*}and
\begin{equation}\label{eq:2.4}
\int_{B_R(x)}\left [\dfrac{|\Delta u|^2}{|z|^{\beta}}+|z|^{\alpha}|u|^{p+1}\right ]\psi^{2m} dz
\le CR^{N-4-\beta-\frac{8+2\alpha+2\beta}{p-1}}.
\end{equation}Here
\begin{eqnarray*}
& \mathfrak{Q}(\psi^m)=|\nabla \psi|^4+\psi^{2(2-m)}\Big [|\nabla (\Delta \psi^m)\cdot \nabla \psi^m |+|\Delta \psi^m|^2+
\left |\Delta |\nabla \psi^m|^2 \right | \Big ],& \\[0.1cm]
& \mathfrak{R}(\psi^m)=\psi^{2(2-m)}
\Big [\left |\Delta \psi^m \right |\left |x \cdot \nabla \psi^m \right |+
\left |\nabla \psi^m \right |^2+\left |x\cdot \nabla
(|\nabla \psi^m|^2)\right | \Big ].&
\end{eqnarray*}
\end{lemma}

\noindent {\it proof.}
From the definition of a stable solution $u$, it implies that if we take arbitrarily $\zeta \in C_0^4(\mathbb{R}^N)$,
then we obtain
\begin{equation}\label{eq:2.5}
\int_{\mathbb{R}^N}|x|^{\alpha}|u|^{p-1}u\zeta dx=\int_{\mathbb{R}^N} \dfrac{\Delta u}{|x|^{\beta}}\Delta \zeta dx,
\end{equation}and
\begin{equation}\label{eq:2.6}
p\int_{\mathbb{R}^N}|x|^{\alpha}|u|^{p-1}\zeta^2 dx\le \int_{\mathbb{R}^N}\dfrac{|\Delta \zeta|^2}{|x|^{\beta}}dx.
\end{equation}Now, in (\ref{eq:2.5}), we choose $\zeta=u\psi^2$ with $\psi\in C_0^4({\mathbb{R}^N})$, and find
\begin{equation}\label{eq:2.7}
\int_{\mathbb{R}^N}|x|^{\alpha} |u|^{p+1} \psi^2 dx=\int_{\mathbb{R}^N}\dfrac{\Delta u}{|x|^{\beta}}
\Delta (u\psi^2)dx.
\end{equation}We insert the test function $\zeta=u\psi$ into (\ref{eq:2.6}) and get
\begin{equation*}
p \int_{\mathbb{R}^N}|x|^{\alpha}|u|^{p+1} \psi^2 dx\le \int_{\mathbb{R}^N}\dfrac{[\Delta (u\psi)]^2}{|x|^{\beta}} dx.
\end{equation*}Putting the above inequality and (\ref{eq:2.7}) back into (\ref{eq:2.1}) yields
\begin{align*}
(p-1)\int_{\mathbb{R}^N}& |x|^{\alpha}|u|^{p+1}\psi^2 dx \le
\int_{\mathbb{R}^N}\dfrac{1}{|x|^{\beta}} \Big [4(\nabla u\cdot \nabla \psi)^2-2u\Delta u|\nabla \psi|^2\Big ]dx\\[0.08cm]
& +\int_{\mathbb{R}^N}\dfrac{u^2}{|x|^{\beta}} \Big [2|\nabla (\Delta \psi)\cdot \nabla \psi|+|\Delta \psi|^2
+2\beta |x|^{-2}\Delta \psi (x\cdot \nabla \psi) \Big] dx.
\end{align*}Combining with the identity (\ref{eq:2.2}), we have
\begin{align}\label{eq:2.8}
\int_{\mathbb{R}^N}|x|^{\alpha} & |u|^{p+1} \psi^2 dx \le C \int_{\mathbb{R}^N} \dfrac{|u\Delta u|}{|x|^{\beta}} |\nabla \psi|^2 dx\nonumber \\[0.1cm]
& +C \int_{\mathbb{R}^N} \dfrac{u^2}{|x|^{\beta}} \Big [|\nabla (\Delta \psi)\cdot \nabla \psi|+|\Delta \psi|^2+\left |\Delta |\nabla \psi|^2 \right |\Big ]dx \nonumber \\[0.1cm]
&+ C\int_{\mathbb{R}^N} \dfrac{u^2}{|x|^{\beta+2}}\Big [|\Delta \psi||x \cdot \nabla \psi|+|\nabla \psi|^2+\left |x\cdot \nabla
(|\nabla \psi|^2)\right | \Big ]dx.
\end{align}Since $\Delta (u\psi)=\Delta u \psi+2\nabla u \cdot \nabla \psi+u\Delta \psi$, it implies from (\ref{eq:2.7}), (\ref{eq:2.8}) and Lemma \ref{eq:l2.2}, that
\begin{align}\label{eq:2.9}
\int_{\mathbb{R}^N}\dfrac{|\Delta u|^2}{|x|^{\beta}} \psi^2dx \le & C \int_{\mathbb{R}^N}\dfrac{|u\Delta u|}{|x|^{\beta}} |\nabla \psi|^2 dx \nonumber  +C \int_{\mathbb{R}^N} \dfrac{u^2}{|x|^{\beta}} \Big [|\nabla (\Delta \psi)\cdot \nabla \psi|+|\Delta \psi|^2+\left |\Delta |\nabla \psi|^2 \right |\Big ]dx \nonumber \\[0.1cm]
&+ C\int_{\mathbb{R}^N} \dfrac{u^2}{|x|^{\beta+2}}\Big [|\Delta \psi||x \cdot \nabla \psi|+|\nabla \psi|^2+\left |x\cdot \nabla
(|\nabla \psi|^2)\right | \Big ]dx.
\end{align}\vskip .05in

Replace $\psi$ by $\psi^m$ in (\ref{eq:2.8}) and (\ref{eq:2.9}) with $m>2$ to lead to
\begin{align*}
\int_{\mathbb{R}^N}&\left [\dfrac{|\Delta u|^2}{|x|^{\beta}} +|x|^{\alpha}|u|^{p+1}\right ]\psi^{2m} dx \le C\int_{\mathbb{R}^N}
\dfrac{|u\Delta u|}{|x|^{\beta}} \psi^{2(m-1)} |\nabla \psi|^2 dx\\[0.1cm]
& +C\int_{\mathbb{R}^N}\dfrac{u^2}{|x|^{\beta}} \Big [|\nabla (\Delta \psi^m)\cdot \nabla \psi^m|+|\Delta \psi^m|^2+\left |
\Delta |\nabla \psi^m|^2 \right | \Big ]dx \\[0.1cm]
& + C\int_{\mathbb{R}^N} \dfrac{u^2}{|x|^{\beta+2}}\Big [\left |\Delta \psi^m \right |\left |x \cdot \nabla \psi^m \right |+
\left |\nabla \psi^m \right |^2+\left |x\cdot \nabla
(|\nabla \psi^m|^2)\right | \Big ]dx.
\end{align*}Utilizing Young's inequality, we obtain
\begin{equation*}
\int_{\mathbb{R}^N} \dfrac{|u\Delta u|}{|x|^{\beta}}\psi^{2(m-1)}|\nabla \psi|^2 dx\le \dfrac{1}{2C} \int_{\mathbb{R}^N}
\dfrac{|\Delta u|^2}{|x|^{\beta}} \psi^{2m} dx +C\int_{\mathbb{R}^N} \dfrac{u^2}{|x|^{\beta}}\psi^{2(m-2)}|\nabla \psi|^4dx.
\end{equation*}Thus, it implies
\begin{align*}
\int_{\mathbb{R}^N}\left [\dfrac{|\Delta u|^2}{|x|^{\beta}}+|x|^{\alpha}|u|^{p+1}\right ] & \psi^{2m} dx
\le C\int_{\mathbb{R}^N} \dfrac{u^2}{|x|^{\beta}} \psi^{2(m-2)} \mathfrak{Q}(\psi^m)dx \\[0.1cm]
& +C \int_{\mathbb{R}^N} \dfrac{u^2}{|x|^{\beta+2}} \psi^{2(m-2)} \mathfrak{R}(\psi^m) dx,
\end{align*}where $\mathfrak{Q}(\psi^m)=|\nabla \psi|^4+\psi^{2(2-m)}\Big [|\nabla (\Delta \psi^m)\cdot \nabla \psi^m |+|\Delta \psi^m|^2+
\left |\Delta |\nabla \psi^m|^2 \right | \Big ]$ and $\mathfrak{R}(\psi^m)=\psi^{2(2-m)}
\Big [\left |\Delta \psi^m \right |\left |x \cdot \nabla \psi^m \right |+
\left |\nabla \psi^m \right |^2+\left |x\cdot \nabla
(|\nabla \psi^m|^2)\right | \Big ]$. Taking $(m-2)(p+1)\ge 2m$,
we use H\"{o}lder's inequality to the both terms in the right hand side of the above inequality and get
\begin{align*}
\int_{\mathbb{R}^N} \dfrac{u^2}{|x|^{\beta}} \psi^{2(m-2)} \mathfrak{Q}(\psi^m)dx =\int_{\mathbb{R}^N} |x|^{\frac{2\alpha}{p+1}}u^2 \psi^{2(m-2)}|x|^{-\frac{2\alpha}{p+1}-\beta}\mathfrak{Q}(\psi^m)dx \\[0.1cm]
\le \left (\int_{\mathbb{R}^N}|x|^{\alpha}|u|^{p+1}\psi^{2m}dx \right )^{\frac{2}{p+1}}
\left (\int_{\mathbb{R}^N}|x|^{-\frac{2\alpha+\beta p+\beta}{p-1}}\mathfrak{Q}(\psi^m)^{\frac{p+1}{p-1}}dx \right )^{\frac{p-1}{p+1}},
\end{align*}and
\begin{align*}
\int_{\mathbb{R}^N} \dfrac{u^2}{|x|^{\beta+2}} \psi^{2(m-2)} \mathfrak{R}(\psi^m)dx =\int_{\mathbb{R}^N} |x|^{\frac{2\alpha}{p+1}}u^2 \psi^{2(m-2)}|x|^{-\frac{2\alpha}{p+1}-\beta-2}\mathfrak{R}(\psi^m)dx \\[0.1cm]
\le \left (\int_{\mathbb{R}^N}|x|^{\alpha}|u|^{p+1}\psi^{2m}dx \right )^{\frac{2}{p+1}}
\left (\int_{\mathbb{R}^N}|x|^{-\frac{2\alpha+(\beta+2)(p+1)}{p-1}}\mathfrak{R}(\psi^m)^{\frac{p+1}{p-1}}dx \right )^{\frac{p-1}{p+1}}.
\end{align*}
Therefore, we find
\begin{align*}
\int_{\mathbb{R}^N}\left [\dfrac{|\Delta u|^2}{|x|^{\beta}}+|x|^{\alpha}|u|^{p+1}\right ] & \psi^{2m} dx
\le C\int_{\mathbb{R}^N}|x|^{-\frac{2\alpha+\beta p+\beta}{p-1}}\mathfrak{Q}(\psi^m)^{\frac{p+1}{p-1}}dx\\[0.1cm]
& +C\int_{\mathbb{R}^N}|x|^{-\frac{2\alpha+(\beta+2)(p+1)}{p-1}}\mathfrak{R}(\psi^m)^{\frac{p+1}{p-1}}dx.
\end{align*}

Let us choose $\psi \in C_0^4(B_{2R}(x))$ a cut-off function verifying $0 \le \psi \le 1$, $\psi \equiv 1$ in $B_R(x)$, and  $|\nabla^k \psi|
\le \dfrac{C}{R^k}$ for $k\le 3$. Substituting $\psi$ into (\ref{eq:2.8}), (\ref{eq:2.9}) and the above inequality, we have
\begin{align*}
& \int_{B_R(x)}\left (\dfrac{|\Delta u|^2}{|z|^{\beta}} +|z|^{\alpha}|u|^{p+1} \right )dz\\[0.1cm]
& \le CR^{-2} \int_{B_{2R}(x)\backslash B_R(x)}\dfrac{|u\Delta u|}{|z|^{\beta}} dz +CR^{-4}\int_{B_{2R}(x)\backslash B_R(x)}\dfrac{u^2}{|z|^{\beta}} dz.
\end{align*}and
\begin{equation*}
\int_{B_R(x)}\left [\dfrac{|\Delta u|^2}{|z|^{\beta}}+|z|^{\alpha}|u|^{p+1}\right ]\psi^{2m} dz
\le CR^{N-4-\beta-\frac{8+2\alpha+2\beta}{p-1}}. \eqno \square
\end{equation*}\vskip .08in

\begin{remark}\label{eq:r2.1}
If the domain $\mathbb{R}^N$ is replaced by the subset $\Omega$ (boundedness or not) in Lemma \ref{eq:l2.1}-Lemma \ref{eq:al2.3}, then the
conclusions are also true.
\end{remark}

\begin{lemma}{\em (Pohozaev identity)}
Let $u$ be a classical solution of (\ref{eq:1.1}), then we have
\begin{align}\label{eq:2.10}
& \dfrac{N-4-\beta}{2}\int_{\Omega}\dfrac{|\Delta u|^2}{|x|^{\beta}} dx -\dfrac{N+\alpha}{p+1} \int_{\Omega} |x|^{\alpha}|u|^{p+1} dx \nonumber \\[0.08cm]
&=\dfrac{1}{2} \int_{\partial \Omega} \dfrac{|\Delta u|^2}{|x|^{\beta}} (x \cdot \nu) dS
-\dfrac{1}{p+1} \int_{\partial \Omega} |x|^{\alpha}|u|^{p+1}(x \cdot \nu) dS \nonumber \\[0.1cm]
&\;\;\ -\int_{\partial \Omega}\dfrac{\Delta u}{|x|^{\beta}} \nabla (x\cdot \nabla u)\cdot \nu dS+
\int_{\partial \Omega}\nabla \left (\dfrac{\Delta u}{|x|^{\beta}}\right )\cdot \nu (x \cdot \nabla u)dS,
\end{align}where $\nu$ denotes the outward unit normal vector field.
\end{lemma}

\begin{proof}
Multiplying (\ref{eq:1.1}) by $(x \cdot \nabla u)$, we obtain
\begin{equation*}
\Delta (|x|^{-\beta} \Delta u)(x \cdot \nabla u)=|x|^{\alpha} |u|^{p-1}u(x\cdot \nabla u),\quad \mbox{in}\;\; \Omega \backslash \{0\}.
\end{equation*}Hence, for every small $\varepsilon >0$, we have
\begin{equation}\label{eq:2.11}
\int_{\Omega \backslash B_{\varepsilon}}\Delta (|x|^{-\beta} \Delta u)(x \cdot \nabla u) dx
=\int_{\Omega \backslash B_{\varepsilon}} |x|^{\alpha} |u|^{p-1}u(x\cdot \nabla u) dx.
\end{equation}Apply the divergence theorem and integration by parts to calculate the right hand side and
the left hand side of (\ref{eq:2.11}) respectively, and get
\begin{align}\label{eq:2.12}
& \dfrac{1}{p+1} \int_{\Omega \backslash B_{\varepsilon}} |x|^{\alpha} \left (x \cdot \nabla \left (|u|^{p+1}\right )\right ) dx =-\dfrac{N+\alpha}{p+1}\int_{\Omega \backslash B_{\varepsilon}} |x|^{\alpha}|u|^{p+1} dx \nonumber \\[0.2cm]
& +\dfrac{1}{p+1} \int_{\partial \Omega}|x|^{\alpha}|u|^{p+1}(x \cdot \nu)dS-\dfrac{1}{p+1} \int_{\partial B_{\varepsilon}}|x|^{\alpha}|u|^{p+1}(x \cdot \nu)dS,
\end{align}and
\begin{align*}
& \int_{\Omega \backslash B_{\varepsilon}}\Delta (|x|^{-\beta} \Delta u)(x \cdot \nabla u) dx=\sum\limits_{i,j=1}^N\int_{\Omega \backslash B_{\varepsilon}} (|x|^{-\beta}\Delta u)_{x_ix_i}
(x^ju_{x_j})dx \\[0.2cm]
&=\sum\limits_{i,j=1}^N \int_{\Omega \backslash B_{\varepsilon}}(|x|^{-\beta}\Delta u)(x^ju_{x_j})_{x_ix_i} dx-\int_{\partial \Omega} |x|^{-\beta}\Delta u \nabla (x\cdot \nabla u)\cdot \nu dS  \\[0.2cm]
&\quad +\int_{\partial B_{\varepsilon}} |x|^{-\beta}\Delta u \nabla (x\cdot \nabla u)\cdot \nu dS +\int_{\partial \Omega} \nabla (|x|^{-\beta}\Delta u)\cdot \nu (x\cdot \nabla u)dS  \\[0.2cm]
& \quad -\int_{\partial B_{\varepsilon}} \nabla (|x|^{-\beta}\Delta u)\cdot \nu (x\cdot \nabla u)dS.
\end{align*}Again computing the first term in the right hand side of the above equality yields
\begin{align*}
& \sum\limits_{i,j=1}^N \int_{\Omega \backslash B_{\varepsilon}}(|x|^{-\beta}\Delta u)(x^ju_{x_j})_{x_ix_i} dx =\sum\limits_{i,j=1}^N
\int_{\Omega \backslash B_{\varepsilon}} |x|^{-\beta}\Delta u \Big [2\delta^{ij}u_{x_ix_j} +x^ju_{x_ix_ix_j}\Big ]dx\\[0.2cm]
&=-\dfrac{N-4-\beta}{2}\int_{\Omega \backslash B_{\varepsilon}}|x|^{-\beta}|\Delta u|^2 dx+\int_{\partial \Omega} \dfrac{|\Delta u|^2}{2} |x|^{-\beta} x\cdot \nu dS-\int_{\partial B_{\varepsilon}} \dfrac{|\Delta u|^2}{2} |x|^{-\beta} x\cdot \nu dS,
\end{align*}and putting back into the above equality leads to
\begin{align}\label{eq:2.13}
& \int_{\Omega \backslash B_{\varepsilon}}\Delta (|x|^{-\beta} \Delta u)(x \cdot \nabla u) dx=-\dfrac{N-4-\beta}{2}\int_{\Omega \backslash B_{\varepsilon}}\dfrac{|\Delta u|^2}{|x|^{\beta}} dx+\int_{\partial \Omega} \dfrac{|\Delta u|^2}{2} |x|^{-\beta} x\cdot \nu dS \nonumber \\[0.2cm]
&\quad -\int_{\partial \Omega} \dfrac{\Delta u}{|x|^{\beta}} \nabla (x\cdot \nabla u)\cdot \nu dS
+\int_{\partial \Omega} \nabla \left (\dfrac{\Delta u}{|x|^{\beta}} \right )\cdot \nu (x\cdot \nabla u)dS-\int_{\partial B_{\varepsilon}} \dfrac{|\Delta u|^2}{2} |x|^{-\beta} x\cdot \nu dS \nonumber \\[0.2cm]
& \quad +\int_{\partial B_{\varepsilon}} \dfrac{\Delta u}{|x|^{\beta}} \nabla (x\cdot \nabla u)\cdot \nu dS
-\int_{\partial B_{\varepsilon}} \nabla \left (\dfrac{\Delta u}{|x|^{\beta}} \right )\cdot \nu (x\cdot \nabla u)dS.
\end{align}Since $u\in C^4(\Omega)$, $\alpha>-4$ and $0 \le \beta \le \dfrac{N-4}{2}$, the above integrations are well-defined. Now, we insert (\ref{eq:2.12}) and (\ref{eq:2.13}) into (\ref{eq:2.11}), take $\varepsilon \to 0$ and pass to the limit to obtain the identity (\ref{eq:2.10}).
\end{proof}

Inspired by the ideas of \cite{Davila,Du,Hu}, we will apply {\it Pohozaev identity} to construct a monotonicity formula
which is a crucial tool. More precisely, choose $u\in W^{4,2}_{loc}(\Omega)$ and $|x|^{\alpha}|u|^{p+1} \in L_{loc}^1(\Omega)$,
fix $x\in \Omega$, let $0<r<R$ and $B_r(x) \subset B_R(x) \subset \Omega$, and define
\begin{align}\label{eq:2.14}
\mathcal{M}(r;x,u)=& r^{\delta} \int_{B_r(x)} \dfrac{1}{2}\dfrac{|\Delta u|^2}{|z|^{\beta}}-\dfrac{1}{p+1}|z|^{\alpha}|u|^{p+1} \nonumber \\[0.1cm]
& +\dfrac{(1+\beta)\lambda}{2}(N-2-\lambda) \left (r^{2\lambda+1-N}\int_{\partial B_r(x)} u^2 \right ) \nonumber \\[0.1cm]
& +\dfrac{\lambda}{2} (N-2-\lambda) \dfrac{d}{dr} \left (r^{2\lambda+2-N}
\int_{\partial B_r(x)} u^2 \right ) \nonumber \\[0.1cm]
&+\frac{r^3}{2}\dfrac{d}{dr} \left [r^{2\lambda+1-N}\int_{\partial B_r(x)} \left (\lambda r^{-1}u+\dfrac{\partial u}{\partial r}\right )^2 \right ] \nonumber \\[0.1cm]
& +\dfrac{1+\beta-\lambda}{2}r^{2\lambda +3-N} \int_{\partial B_r(x)} \left (|\nabla u|^2-\left |\dfrac{\partial u}{\partial r}\right |^2 \right ) \nonumber \\[0.1cm]
& +\dfrac{1}{2}\dfrac{d}{dr} \left [ r^{2\lambda+4-N} \int_{\partial B_r(x)} \left (|\nabla u|^2-\left |\dfrac{\partial u}{\partial r}\right |^2 \right ) \right ].
\end{align}Here and in the following, we always set $\delta:=\dfrac{8+2\alpha+2\beta}{p-1}+4+\beta-N$ and $\lambda:=\dfrac{4+\alpha+\beta}{p-1}$.\vskip .06in

\begin{theorem}\label{eq:t2.1}
Let $p \ge \dfrac{N+4+2\alpha+\beta}{N-4-\beta}$ and let $u\in W^{2,2}_{loc}(\Omega)$,
$|x|^{\alpha} |u|^{p+1} \in L_{loc}^1(\Omega)$ and $|x|^{-\beta}|\Delta u|^2 \in L_{loc}^1(\Omega)$
be a weak solution of (\ref{eq:1.1}). Then $\mathcal{M}(r;x,u)$ is nondecreasing in $r \in (0,R)$ and satisfies the inequality
\begin{equation*}
\dfrac{d}{dr}\mathcal{M}(r;0,u) \ge C(N,p, \alpha, \beta) r^{-N+2+2\lambda} \int_{\partial B_r}
\left (\lambda r^{-1}u+\dfrac{\partial u}{\partial r} \right )^2 dS,
\end{equation*}where $C(N,p,\alpha,\beta)=(N-2)(2+\beta)+2\lambda(N-4-\beta-\lambda)-\dfrac{\beta^2}{8}>0$.\vskip .06in

Furthermore, if $\mathcal{M}(r;0,u)\equiv \mbox{const}$, for all $r \in (0,R)$, then $u$ is
homogeneous in $B_R\backslash \{0\}$, i.e., $\forall \mu \in (0,1]$, $x \in B_R\backslash \{0\}$,
\begin{equation*}
u(\mu x)=\mu^{-\frac{4+\alpha+\beta}{p-1}}u(x).
\end{equation*}
\end{theorem}

\noindent {\it Proof.}
Define a function by
\begin{equation*}
\mathcal{F}(\kappa):=\kappa^{\delta} \int_{B_{\kappa}} \left (\dfrac{1}{2}\dfrac{|\Delta u|^2}{|x|^{\beta}}-\dfrac{1}{p+1}|x|^{\alpha}|u|^{p+1} \right )dx.
\end{equation*}Differentiating the function $\mathcal{F}(\kappa)$ in $\kappa$ arrives at
\begin{align}\label{eq:2.15}
\dfrac{d\mathcal{F}(\kappa)}{d\kappa} = &\delta \kappa^{\delta-1}\int_{B_{\kappa}} \left (\dfrac{1}{2}\dfrac{|\Delta u|^2}{|x|^{\beta}}-\dfrac{1}{p+1}|x|^{\alpha}|u|^{p+1} \right )dx \nonumber \\[0.1cm]
&+\kappa^{\delta }\int_{\partial B_{\kappa}} \left (\dfrac{1}{2}\dfrac{|\Delta u|^2}{|x|^{\beta}}-\dfrac{1}{p+1}|x|^{\alpha}|u|^{p+1} \right )dS.
\end{align}Multiply the equation (\ref{eq:1.1}) by $u$ and
integrate by parts to get
\begin{align*}
\int_{\Omega}|x|^{\alpha}|u|^{p+1} dx & =-\int_{\Omega} \nabla \left (\dfrac{\Delta u}{|x|^{\beta}} \right )\cdot \nabla u dx+\int_{\partial \Omega}\dfrac{\partial }{\partial \nu}
\left (\dfrac{\Delta u}{|x|^{\beta}} \right ) udS\\[0.1cm]
&=\int_{\Omega}\dfrac{|\Delta u|^2}{|x|^{\beta}} dx-\int_{\partial \Omega}\dfrac{\Delta u}{|x|^{\beta}}(\nabla u\cdot \nu )dS+\int_{\partial \Omega}\dfrac{\partial }{\partial \nu}
\left (\dfrac{\Delta u}{|x|^{\beta}} \right ) udS,
\end{align*}implies
\begin{equation}\label{eq:2.16}
\int_{\Omega}\dfrac{|\Delta u|^2}{|x|^{\beta}} dx-\int_{\Omega}|x|^{\alpha}|u|^{p+1} dx=\int_{\partial \Omega}\dfrac{\Delta u}{|x|^{\beta}}(\nabla u\cdot \nu )dS-\int_{\partial \Omega}\dfrac{\partial }{\partial \nu}
\left (\dfrac{\Delta u}{|x|^{\beta}} \right ) udS.
\end{equation}In addition, it is easily to see that
\begin{align*}
&\delta \kappa^{\delta-1}\int_{B_{\kappa}} \left (\dfrac{1}{2}\dfrac{|\Delta u|^2}{|x|^{\beta}}-\dfrac{1}{p+1}|x|^{\alpha}|u|^{p+1} \right )dx\\[0.15cm]
&=\lambda \kappa^{\delta-1} \int_{B_{\kappa}} \left (\dfrac{|\Delta u|^2}{|x|^{\beta}}  -|x|^{\alpha}|u|^{p+1}\right ) dx \;\;\ -\kappa^{\delta-1}\int_{B_{\kappa}}\left ( \dfrac{N-4-\beta}{2}\dfrac{|\Delta u|^2}{|x|^{\beta}} -\dfrac{N+\alpha}{p+1}  |x|^{\alpha}|u|^{p+1} \right ) dx.
\end{align*}Therefore, combining the result with (\ref{eq:2.10}), (\ref{eq:2.15}) and (\ref{eq:2.16}), we obtain
\begin{align}\label{eq:2.17}
& \dfrac{d\mathcal{F}(\kappa)}{d\kappa}= \lambda \kappa^{\delta -1} \left [\int_{\partial B_{\kappa}}\dfrac{\Delta u}{|x|^{\beta}}(\nabla u\cdot \nu )dS-\int_{\partial B_{\kappa}}\dfrac{\partial }{\partial \nu} \left (\dfrac{\Delta u}{|x|^{\beta}} \right ) udS \right ] \nonumber \\[0.1cm]
& \; \ + \kappa^{\delta -1}\int_{\partial B_{\kappa}} \dfrac{\Delta u}{|x|^{\beta}} \nabla (x\cdot \nabla u)\cdot \nu dS -\kappa^{\delta -1}
\int_{\partial B_{\kappa}}\nabla \left (\dfrac{\Delta u}{|x|^{\beta}} \right )\cdot \nu (x \cdot \nabla u)dS.
\end{align}\vskip .05in

Denote $u^{\kappa}(x):=\kappa^{\frac{4+\alpha+\beta}{p-1}}u(\kappa x)$. Now, computing the first term in the right hand side of (\ref{eq:2.17}) leads to
\begin{align}\label{eq:2.18}
\lambda \kappa^{\delta-1}\int_{\partial B_{\kappa}} |x|^{-\beta} \Delta u (\nabla u \cdot \nu)dS &=\lambda \kappa^2 \int_{\partial B_R}
\Delta \kappa^{\frac{4+\alpha+\beta}{p-1}} u \left (\nabla \kappa^{\frac{4+\alpha+\beta}{p-1}} u \cdot \nu \right ) dS \cdot \kappa^{1-N} \nonumber \\[0.1cm]
&= \dfrac{\lambda}{\kappa} \int_{\partial B_1} \Delta u^{\kappa} (\nabla u^{\kappa} \cdot \nu)d \sigma.
\end{align}Similarly,  we calculate the second term in the right hand side of (\ref{eq:2.17}) and get
\begin{align}\label{eq:2.19}
& \lambda \kappa^{\delta-1}\int_{\partial B_{\kappa}}\dfrac{\partial }{\partial \nu} \left (\dfrac{\Delta u}{|x|^{\beta}} \right ) udS  =
\lambda \kappa^{\delta-1}\int_{\partial B_{\kappa}}|x|^{-\beta} \left [(\nabla (\Delta u)\cdot \nu)-\beta|x|^{-2}\Delta u(x\cdot \nu) \right ] udS \nonumber \\[0.2cm]
& =\lambda \kappa^2 \int_{\partial B_{\kappa}} \left [\left (\nabla \left (\Delta  \kappa^{\frac{4+\alpha+\beta}{p-1}}u\right )\cdot \nu\right )-
\beta \kappa^{-2}\Delta \left ( \kappa^{\frac{4+\alpha+\beta}{p-1}}u\right )(x\cdot \nu) \right ] \kappa^{\frac{4+\alpha+\beta}{p-1}}udS \cdot \kappa^{1-N} \nonumber \\[0.2cm]
& =\dfrac{\lambda}{\kappa}\int_{\partial B_1} \Big [(\nabla (\Delta u^{\kappa}) \cdot \nu) -\beta \Delta u^{\kappa} \Big ] u^{\kappa} d \sigma.
\end{align}Similar to the above calculation, we find
\begin{align}\label{eq:2.20}
& \kappa^{\delta -1}\int_{\partial B_{\kappa}} |x|^{-\beta} \Delta u \nabla (x\cdot \nabla u)\cdot \nu dS =\dfrac{1}{\kappa}\int_{\partial B_1} \Delta u^{\kappa} \nabla \left (x\cdot \nabla u^{\kappa} \right ) \cdot \nu d \sigma,
\end{align}and
\begin{align}\label{eq:2.21}
& \kappa^{\delta -1}
\int_{\partial B_{\kappa}}\nabla (|x|^{-\beta} \Delta u)\cdot \nu (x \cdot \nabla u)dS =\dfrac{1}{\kappa}\int_{\partial B_1} \Big [\left (\nabla \left (\Delta u^{\kappa}\right ) \cdot \nu\right ) -\beta \Delta u^{\kappa}\Big ]\left (x \cdot \nabla u^{\kappa} \right )d \sigma.
\end{align}We use spherical coordinates $r=|x|$, $\theta=\dfrac{x}{|x|}\in \mathbb{S}^{N-1}$ and
write $u^{\kappa}(x)=u^{\kappa}(r,\theta)$, then we insert (\ref{eq:2.18})-(\ref{eq:2.21}) into (\ref{eq:2.17}) to obtain
\begin{align*}
\dfrac{d\mathcal{F}(\kappa)}{d\kappa}=& \dfrac{\lambda}{\kappa}\int_{\partial B_1} \Delta u^{\kappa} (\nabla u^{\kappa} \cdot \nu)d \sigma -\dfrac{\lambda}{\kappa}
\int_{\partial B_1} \Big [(\nabla (\Delta u^{\kappa}) \cdot \nu) -\beta \Delta u^{\kappa} \Big ] u^{\kappa} d \sigma \\[0.1cm]
&+\dfrac{1}{\kappa} \int_{\partial B_1} \Delta u^{\kappa} \nabla \left (x\cdot \nabla u^{\kappa} \right ) \cdot \nu d \sigma
-\dfrac{1}{\kappa}\int_{\partial B_1} \Big [\left (\nabla \left (\Delta u^{\kappa}\right ) \cdot \nu\right ) -\beta \Delta u^{\kappa}\Big ]\left (x \cdot \nabla u^{\kappa} \right )d \sigma\\[0.1cm]
= & \dfrac{1}{\kappa} \int_{\partial B_1} \lambda \left (\dfrac{\partial^2u^{\kappa}}{\partial r^2} +(N-1)\dfrac{\partial u^{\kappa}}{\partial r}
+\Delta_{\theta} u^{\kappa} \right )\dfrac{\partial u^{\kappa}}{\partial r} \\[0.1cm]
& -\lambda \left [\dfrac{\partial^3 u^{\kappa}}{\partial r^3}+(N-1-\beta)\dfrac{\partial^2 u^{\kappa}}{\partial r^2}
-(N-1)(1+\beta)\dfrac{\partial u^{\kappa}}{\partial r}-(2+\beta) \Delta_{\theta} u^{\kappa}\right ]u^{\kappa} \\[0.1cm]
& +\left [\dfrac{\partial^2 u^{\kappa}}{\partial r^2}+(N-1)\dfrac{\partial u^{\kappa}}{\partial r}+\Delta_{\theta} u^{\kappa} \right ]\left (\dfrac{\partial^2 u^{\kappa}}{\partial r^2}+\dfrac{\partial u^{\kappa}}{\partial r} \right )\\[0.1cm]
&- \left [\dfrac{\partial^3 u^{\kappa}}{\partial r^3}+(N-1-\beta)\dfrac{\partial^2 u^{\kappa}}{\partial r^2}
-(N-1)(1+\beta)\dfrac{\partial u^{\kappa}}{\partial r}-(2+\beta) \Delta_{\theta} u^{\kappa}\right ]\dfrac{\partial u^{\kappa}}{\partial r}  \\[0.15cm]
=& \dfrac{1}{\kappa} \int_{\partial B_1}- \dfrac{\partial^3 u^{\kappa}}{\partial r^3}\dfrac{\partial u^{\kappa}}{\partial r}
-\lambda \dfrac{\partial^3 u^{\kappa}}{\partial r^3}u^{\kappa}+\left (\dfrac{\partial^2 u^{\kappa}}{\partial r^2}\right )^2
+(\lambda+1+\beta)\dfrac{\partial^2u^{\kappa}}{\partial r^2}\dfrac{\partial u^{\kappa}}{\partial r}\\[0.1cm]
& +(N-1)(\lambda+2+\beta)\left (\dfrac{\partial u^{\kappa}}{\partial r}\right )^2
-\lambda (N-1-\beta) \dfrac{\partial^2 u^{\kappa}}{\partial r^2}u^{\kappa}
+\lambda (N-1)(1+\beta) \dfrac{\partial u^{\kappa}}{\partial r}u^{\kappa}\\[0.1cm]
& +\dfrac{1}{\kappa}\int_{\partial B_1} \Delta_{\theta} u^{\kappa}\dfrac{\partial ^2 u^{\kappa}}{\partial r^2} +(\lambda+3+\beta)
\Delta_{\theta} u^{\kappa}\dfrac{\partial u^{\kappa}}{\partial r}+\lambda (2+\beta )\Delta_{\theta}u^{\kappa}u^{\kappa}\\[0.1cm]
:= &\; \mathfrak{T}_1+\mathfrak{T}_2,
\end{align*}where $\Delta_{\theta}$ represents the Laplace-Beltrami operator on $\partial B_1$ and
$\nabla_{\theta}$ is the tangential derivative on $\partial B_1$.\vskip 0.05in

Differentiating $u^{\kappa}$ in $\kappa$ implies
\begin{equation}\label{eq:2.22}
\dfrac{du^{\kappa}}{d\kappa}(x)=\dfrac{1}{\kappa} \left [\lambda u^{\kappa}(x)+r\dfrac{\partial u^{\kappa}}{\partial r}(x)\right ]
\Longrightarrow
r \dfrac{\partial u^{\kappa}}{\partial r} =\kappa \dfrac{du^{\kappa}}{d\kappa}-\lambda u^{\kappa}.
\end{equation}
Differentiating the equation (\ref{eq:2.22}) in $\kappa$ and $r$ respectively yields
\begin{align*}
r \dfrac{\partial}{\partial r}\dfrac{du^{\kappa}}{d\kappa}=\kappa \dfrac{d^2u^{\kappa}}{d\kappa^2}+(1-\lambda)\dfrac{du^{\kappa}}{d\kappa}\;\; \mbox{and}\;\;
\kappa \dfrac{\partial }{\partial r}\dfrac{du^{\kappa}}{d\kappa}=& (1+\lambda)\dfrac{\partial u^{\kappa}}{\partial r}+r\dfrac{\partial^2 u^{\kappa}}{\partial r^2}.
\end{align*}Then, combining the above two equalities with (\ref{eq:2.22}), we obtain that, on $\partial B_1$
\begin{equation}\label{eq:2.23}
\dfrac{\partial^2 u^{\kappa}}{\partial r^2}=r^2 \dfrac{\partial^2 u^{\kappa}}{\partial r^2} =
\kappa^2\dfrac{d^2u^{\kappa}}{d\kappa^2}-2\lambda\kappa \dfrac{du^{\kappa}}{d\kappa}+\lambda (1+\lambda ) u^{\kappa}.
\end{equation}Similarly, we find
\begin{align*}
r^2 \dfrac{\partial^3 u}{\partial r^3}+2r\dfrac{\partial^2 u}{\partial r^2} &=\kappa^2\dfrac{\partial}{\partial r}
\dfrac{d^2u^{\kappa}}{d\kappa^2}-2\lambda \kappa \dfrac{\partial}{\partial r}\dfrac{du^{\kappa}}{d\kappa}+\lambda(1+\lambda)\dfrac{\partial u^{\kappa}}{\partial r},\\[0.1cm]
r\dfrac{\partial }{\partial r}\dfrac{d^2 u^{\kappa}}{d\kappa^2}& =\kappa \dfrac{d^3u^{\kappa}}{d\kappa^3}+(2-\lambda)\dfrac{d^2u^{\kappa}}{d\kappa^2}.
\end{align*}Then, on $\partial B_1$, we have
\begin{align}\label{eq:2.24}
\dfrac{\partial^3 u}{\partial r^3}
=& \kappa^3\dfrac{d^3u^{\kappa}}{d\kappa^3} -3\lambda \kappa^2\dfrac{d^2u^{\kappa}}{d\kappa^2} +3\lambda(1+\lambda)\kappa\dfrac{du^{\kappa}}{d\kappa} -\lambda(1+\lambda)(2+\lambda)u^{\kappa} .
\end{align}\vskip 0.05in

Substituting (\ref{eq:2.22}) and (\ref{eq:2.23}) into the expression of $\mathfrak{T}_2$ arrives at
\begin{align*}
\mathfrak{T}_2
= & \int_{\partial B_1} \kappa \Delta_{\theta}u^{\kappa} \dfrac{d^2 u^{\kappa}}{d \kappa^2}-(\lambda-3-\beta) \Delta_{\theta}u^{\kappa} \dfrac{d u^{\kappa}}{d \kappa} \\[0.1cm]
= & \int_{\partial B_1}-\kappa \nabla_{\theta}u^{\kappa}\nabla_{\theta}\dfrac{d^2u^{\kappa}}{d\kappa^2}+(\lambda-3-\beta)\nabla_{\theta}u^{\kappa}
\nabla_{\theta}\dfrac{du^{\kappa}}{d\kappa} \\[0.1cm]
=& -\dfrac{1}{2}\dfrac{d^2}{d\kappa^2}\left [\kappa \int_{\partial B_1} \left |\nabla_{\theta}u^{\kappa} \right |^2 \right ]
+\dfrac{\lambda-1-\beta}{2}\dfrac{d}{d\kappa} \int_{\partial B_1}\left |\nabla_{\theta}u^{\kappa} \right |^2
+\kappa \int_{\partial B_1}\left |\nabla_{\theta}\dfrac{du^{\kappa}}{d\kappa} \right |^2 \\[0.1cm]
\ge & -\dfrac{1}{2}\dfrac{d^2}{d\kappa^2}\left [\kappa \int_{\partial B_1} \left |\nabla_{\theta}u^{\kappa} \right |^2 \right ]
+\dfrac{\lambda -1-\beta}{2}\dfrac{d}{d\kappa} \int_{\partial B_1}\left |\nabla_{\theta}u^{\kappa} \right |^2.
\end{align*}Let us note the two equalities
\begin{align*}
-\kappa^3\dfrac{du^{\kappa}}{d\kappa}\dfrac{d^3u^{\kappa}}{d\kappa^3}= & \dfrac{d}{d\kappa}\left (-\dfrac{\kappa^3}{2}\dfrac{d}{d\kappa}\left (\dfrac{du^{\kappa}}{d\kappa}\right )^2 \right )+3\kappa^2\dfrac{du^{\kappa}}{d\kappa}\dfrac{d^2u^{\kappa}}{d\kappa^2}+\kappa^3\left (\dfrac{d^2u^{\kappa}}{d\kappa^2}\right )^2,\\[0.1cm]
\kappa u^{\kappa}\dfrac{d^2u^{\kappa}}{d\kappa^2}= & \dfrac{d^2}{d\kappa^2}\left (\dfrac{\kappa \left (u^{\kappa}\right )^2}{2} \right )
-2u^{\kappa}\dfrac{du^{\kappa}}{d\kappa}-\kappa \left (\dfrac{du^{\kappa}}{d\kappa} \right )^2.
\end{align*}Inserting (\ref{eq:2.22})-(\ref{eq:2.24}) into the expression of $\mathfrak{T}_1$, and combining with the above two equalities,
we get
\begin{align*}
\mathfrak{T}_1=& \int_{\partial B_1} -\kappa^3 \dfrac{d^3 u^{\kappa}}{d\kappa^3} \dfrac{du^{\kappa}}{d\kappa}+\kappa^3 \left (\dfrac{d^2u^{\kappa}}{d\kappa^2}\right )^2+(1+\beta)\kappa^2 \dfrac{d^2u^{\kappa}}{d\kappa^2}\dfrac{du^{\kappa}}{d\kappa} \\[0.1cm]
& +\Big [(N-1)(2+\lambda+\beta)-\lambda (5+\lambda+2\beta) \Big ]\kappa \left (\dfrac{du^{\kappa}}{d\kappa}\right )^2 \\[0.1cm]
& + \lambda(2+\lambda-N)\kappa u^{\kappa} \dfrac{d^2u^{\kappa}}{d\kappa^2}
+\lambda(3+\beta)(\lambda+2-N) u^{\kappa}\dfrac{du^{\kappa}}{d\kappa} \\[0.15cm]
=& \int_{\partial B_1}\dfrac{d}{d\kappa} \left (-\dfrac{\kappa^3}{2}\dfrac{d}{d\kappa} \left (\dfrac{du^{\kappa}}{d\kappa} \right )^2 \right )
+2\kappa^3 \left (\dfrac{d^2u^{\kappa}}{d\kappa^2} \right )^2+(4+\beta)\kappa^2 \dfrac{d^2u^{\kappa}}{d\kappa^2}\dfrac{du^{\kappa}}{d\kappa} \\[0.1cm]
& +\Big [(N-1)(2+\beta)+2\lambda(N-4-\beta-\lambda) \Big ]\kappa \left (\dfrac{du^{\kappa}}{d\kappa}\right )^2 \\[0.1cm]
& +\lambda (2+\lambda-N) \dfrac{d^2}{d\kappa^2} \left (\dfrac{\kappa (u^{\kappa})^2}{2}\right )
+\lambda(2+\lambda-N)(1+\beta)u^{\kappa}\dfrac{du^{\kappa}}{d\kappa}
\end{align*}
\begin{align*}
\ge & \int_{\partial B_1}\dfrac{d}{d\kappa} \left (-\dfrac{\kappa^3}{2}\dfrac{d}{d\kappa} \left (\dfrac{du^{\kappa}}{d\kappa} \right )^2 \right )
+\dfrac{\lambda (2+\lambda-N)}{2} \dfrac{d^2}{d\kappa^2} \left (\kappa (u^{\kappa})^2\right ) \\[0.1cm]
& +\dfrac{\lambda}{2}(2+\lambda-N)(1+\beta)\dfrac{d}{d\kappa}(u^{\kappa})^2.
\end{align*}Since $p\ge \dfrac{N+4+2\alpha+\beta}{N-4-\beta}$, the deleted terms of $\mathfrak{T}_1$ satisfies
\begin{align*}
2\kappa^3 \left (\dfrac{d^2u^{\kappa}}{d\kappa^2} \right )^2+ & (4+\beta)\kappa^2 \dfrac{d^2u^{\kappa}}{d\kappa^2}\dfrac{du^{\kappa}}{d\kappa}
+\Big [(N-1)(2+\beta)+2\lambda(N-4-\beta-\lambda) \Big ]\kappa \left (\dfrac{du^{\kappa}}{d\kappa}\right )^2 \\[0.15cm]
&=2\kappa\left [ \kappa \dfrac{d^2u^{\kappa}}{d\kappa^2} +\left (1+\dfrac{\beta}{4}\right ) \dfrac{du^{\kappa}}{d\kappa} \right ]^2
+C(N,p,\alpha,\beta) \kappa \left (\dfrac{du^{\kappa}}{d\kappa}\right )^2\ge 0,
\end{align*}where
\begin{align*}
C(N,p,\alpha,\beta)= & (N-1)(2+\beta)+2\lambda(N-4-\beta-\lambda) -2\left (1+\dfrac{\beta}{4}\right )^2\\
= & (N-2)(2+\beta)+2\lambda(N-4-\beta-\lambda)-\dfrac{\beta^2}{8}> 0.
\end{align*}

Now, we rescale and write those $\kappa$ derivatives in $\mathfrak{T}_1$ and $\mathfrak{T}_2$ as follows.
\begin{align*}
\int_{\partial B_1}\dfrac{d }{d \kappa}\left (u^{\kappa}\right )^2 & =\dfrac{d}{d \kappa} \left (\kappa^{2\lambda+1-N}\int_{\partial B_{\kappa}} u^2 \right ),\\[0.1cm]
\int_{\partial B_1} \dfrac{d^2 }{d \kappa^2} \left [\kappa \left (u^{\kappa}\right )^2 \right ] & =\dfrac{d^2}{d\kappa^2} \left (
\kappa^{2\lambda+2-N}\int_{\partial B_{\kappa}} u^2 \right ),\\[0.1cm]
\int_{\partial B_1} \dfrac{d}{d \kappa}\left [\kappa^3 \dfrac{d }{d \kappa}\left ( \dfrac{d u^{\kappa}}{d \kappa}\right )^2 \right ] &=
\dfrac{d}{d \kappa} \left [\kappa^3 \dfrac{d}{d\kappa} \left ( \kappa^{2\lambda+1-N}\int_{\partial B_{\kappa}}
\left (\lambda \kappa^{-1} u +\dfrac{\partial u}{\partial r} \right )^2\right ) \right ],\\[0.1cm]
\dfrac{d}{d \kappa} \left ( \int_{\partial B_1}\left | \nabla_{\theta}u^{\kappa} \right |^2 \right )
& =\dfrac{d}{d\kappa} \left [\kappa^{ 2\lambda+3-N}\int_{\partial B_{\kappa}}\left (|\nabla u|^2-\left |\dfrac{\partial u}{\partial r} \right |^2
\right ) \right ], \\[0.1cm]
\dfrac{d^2}{d \kappa^2} \left (\kappa \int_{\partial B_1}\left | \nabla_{\theta}u^{\kappa} \right |^2 \right )
& =\dfrac{d^2}{d\kappa^2} \left [\kappa^{ 2\lambda+4-N}\int_{\partial B_{\kappa}}\left (|\nabla u|^2-\left |\dfrac{\partial u}{\partial r} \right |^2
\right ) \right ].
\end{align*}Substituting these terms into $\dfrac{d\mathcal{F}(\kappa)}{d \kappa}$ yields
\begin{align*}
\dfrac{d \mathcal{F}(\kappa)}{d\kappa} \ge & \dfrac{\lambda(2+\lambda-N)(1+\beta)}{2}\dfrac{d}{d \kappa} \left (\kappa^{2\lambda+1-N}\int_{\partial B_{\kappa}} u^2 \right )\\[0.1cm]
&+ \dfrac{\lambda(2+\lambda-N)}{2}\dfrac{d^2}{d\kappa^2} \left (
\kappa^{2\lambda+2-N}\int_{\partial B_{\kappa}} u^2 \right )\\[0.1cm]
&-\dfrac{1}{2}\dfrac{d}{d \kappa} \left [\kappa^3 \dfrac{d}{d\kappa} \left ( \kappa^{2\lambda+1-N}\int_{\partial B_{\kappa}}
\left (\lambda \kappa^{-1} u +\dfrac{\partial u}{\partial r} \right )^2\right ) \right ]\\[0.1cm]
& +\dfrac{\lambda-1-\beta}{2} \dfrac{d}{d\kappa} \left [\kappa^{ 2\lambda+3-N}\int_{\partial B_{\kappa}}\left (|\nabla u|^2-\left |\dfrac{\partial u}{\partial r} \right |^2
\right ) \right ]\\[0.1cm]
& - \dfrac{1}{2}\dfrac{d^2}{d\kappa^2} \left [\kappa^{ 2\lambda+4-N}\int_{\partial B_{\kappa}}\left (|\nabla u|^2-\left |\dfrac{\partial u}{\partial r} \right |^2
\right ) \right ].
\end{align*}Applying the properties of integration, we conclude that $\mathcal{M}(r;x,u)$ is well defined and non-decreasing in $r \in (0,R)$.\vskip .06in

Next, we let $\mathcal{M}(r;0,u)\equiv \mbox{const}$, for all $r\in (0,R)$, then, for any $r_1,r_2 \in (0,R)$ with $r_1<r_2$, we have
\begin{align*}
0 & =\mathcal{M}(r_2;0,u)-\mathcal{M}(r_1;0,u) = \int_{r_1}^{r_2} \dfrac{d}{d \mu}\mathcal{M}(\mu;0,u)d \mu \\
& \ge C(N,p,\alpha,\beta) \int_{B_{r_2}\backslash B_{r_1}} |x|^{2+2\lambda -N}\left
(\lambda \mu^{-1}u+\dfrac{\partial u}{\partial \mu} \right )^2 dx.
\end{align*}Thus, we get
\begin{equation*}
\lambda \mu^{-1}u+\dfrac{\partial u}{\partial \mu}=0,\quad \mbox{a.e.}\quad \mbox{in}\; B_R\backslash \{0\}.
\end{equation*}Integrating in $r$ shows that
\begin{equation*}
u(\mu x)=\mu^{-\frac{4+\alpha+\beta}{p-1}}u(x),\;\ \forall \mu \in (0,1],\; x \in B_R\backslash \{0\}. \eqno \square
\end{equation*}

\begin{remark}\label{eq:r2.2}
From the proof of Theorem \ref{eq:t2.1}, we can find that if the linear combination of Pohozaev identity and the identity (\ref{eq:2.16})
minus some terms of $\mathfrak{T}_1$ and $\mathfrak{T}_2$, then it {\it is equivalent to} the derivative form of the monotonicity formula (\ref{eq:2.14}).
\end{remark}

\vskip .2in

\section{Proof of Theorem \ref{eq:t1.1}}

\vskip .1in

First, we give the expression of $N_{\alpha,\beta}(p)$. Now, we define four functions by
\begin{eqnarray*}
& \mathfrak{f}(N):=p\dfrac{4+\alpha+\beta}{p-1}\left (\dfrac{4+\alpha+\beta p}{p-1}+2 \right ) \left (N-2-\dfrac{4+\alpha+\beta}{p-1} \right )
\left (N-4-\dfrac{4+\alpha+\beta p}{p-1} \right ), & \\[0.15cm]
& \mathfrak{g}(N):=p\left (\dfrac{4+\alpha+\beta p}{p-1}+2 \right ) \left (N-4- \dfrac{4+\alpha+\beta p}{p-1}\right )+p\dfrac{4+\alpha+\beta}{p-1}
\left (N-2-\dfrac{4+\alpha+\beta}{p-1} \right ), & \\[0.15cm]
& \mathfrak{F}(N):  =\dfrac{(N+\beta)^2(N-4-\beta)^2}{16}, & \\[0.15cm]
& \mathfrak{G}(N):  =\dfrac{(N+\beta)(N-4-\beta)}{2}.&
\end{eqnarray*}Differentiating the functions $\mathfrak{f}(N)$ and $\mathfrak{F}(N)$ in $N$, we obtain
\begin{eqnarray*}
& \mathfrak{f}'(N)=p\dfrac{4+\alpha+\beta}{p-1}\left (2+\dfrac{4+\alpha+\beta p}{p-1} \right )\left (2N-6-\beta-\dfrac{8+2\alpha+2\beta}{p-1} \right ),& \\[0.15cm]
& \mathfrak{F}'(N)=\dfrac{1}{4}(N+\beta)(N-2)(N-4-\beta).&
\end{eqnarray*}A simple computation yields
\begin{align*}
\mathfrak{f} & (4+\beta+2\lambda)-\mathfrak{F}(4+\beta+2\lambda)=(p-1) \lambda^2 (2+\beta+\lambda)^2 >0,\\[0.15cm]
\mathfrak{f} & (4+\beta+(4p+1)\lambda)-\mathfrak{F}(4+\beta+(4p+1)\lambda) \\[0.1cm]
& =4p^2\lambda^2(2+\beta+\lambda)(2+\beta+4p\lambda)
-\dfrac{(4p+1)^2\lambda^2}{16}\Big [(4+2\beta)+(4p+1)\lambda \Big ]^2\\[0.1cm]
&  <0,\\[0.15cm]
\mathfrak{g} & (4+\beta+2\lambda)-\mathfrak{G}(4+\beta+2\lambda)=2(p-1)\lambda (2+\beta+\lambda) >0,\\[0.15cm]
\mathfrak{f} & ' (4+\beta+2\lambda)-\mathfrak{F}'(4+\beta+2\lambda) =(p-1)\lambda (2+\beta+\lambda)(2+\beta+2\lambda)
>0.
\end{align*}Therefore, we take the least real root $N(p,\alpha,\beta)$ of the following algebra equation between $4+\beta+2\lambda$ with $4+\beta+(4p+1)\lambda$
\begin{align*}
& (p^4-4p^3+6p^2-4p+1)y^4-(8p^4-32p^3+48p^2-32p+8)y^3\\[0.1cm]
& -(p^2-2p+1)\Big [(32\alpha +104\beta +16 \alpha\beta +18\beta^2+112) p^2+(16\alpha \beta  \\[0.1cm]
& +16 \alpha^2 -4\beta^2 +16\beta +96\alpha +160) p
+8\beta +2\beta^2 -16 \Big ]y^2\\[0.1cm]
& +\Big \{ \Big [\Big (48+44\beta +12 \beta^2 +8\alpha \beta +12\alpha +\alpha \beta^2+\beta^3 \Big ) p^2+\Big (64 +56 \alpha +28 \alpha\beta  \\[0.1cm]
& +10\alpha^2 +40\beta +10\beta^2  +4\alpha\beta^2 +3\alpha^2\beta +\beta^3\Big ) p +28\alpha+16+14\alpha^2 +2\alpha^3 \\[0.1cm]
& +12\beta+12\alpha\beta  +3\alpha^2\beta +2\beta^2+\alpha\beta^2 \Big ]16 (p^2-p)-(p-1)^4(32 \beta+8\beta^2)\Big \} y\\[0.1cm]
&   +(p-1)^4\beta^2(\beta+4)^2-16\Big [(8+(2+\beta)\alpha+(6+\beta)\beta)p +(6+\alpha+\beta)\alpha+2\beta+8 \Big ] \\[0.1cm]
& \times \Big [(8+2\beta) p^2+(6\beta +6\alpha +\alpha\beta +\beta^2+8) p+2\alpha +\alpha^2+\alpha \beta \Big ] p\\[0.1cm]
& =0.
\end{align*}

Define
\begin{equation}\label{eq:3.1}
N_{\alpha,\beta}(p):=N(p,\alpha,\beta).
\end{equation}Then, for any $4+\beta+\dfrac{8+2\alpha+2\beta}{p-1}<N<N_{\alpha,\beta}(p)$, we find
\begin{equation}\label{eq:3.2}
\mathfrak{f}(N) >\dfrac{(N+\beta)^2(N-4-\beta)^2}{16}.
\end{equation}Furthermore, combining the above inequality with the inequality $a+b\ge 2\sqrt{ab}$, for all $a,b\ge 0$, we have
\begin{equation}\label{eq:3.3}
\mathfrak{g}(N) >\dfrac{(N+\beta)(N-4-\beta)}{2}.
\end{equation}On the other hand, we easily check that the equality $\mathfrak{f}(N)-\mathfrak{F}(N)>0$ holds, if one of the following conditions  holds:
\begin{itemize}
\item [\rm (i)] $\alpha=\beta$ and $4+\alpha+\dfrac{8+4\alpha}{p-1}<N< 8+3\alpha+\dfrac{8+4\alpha}{p-1}$; or
\item [\rm (ii)] $\alpha=\beta=0$ and
$4+\dfrac{8}{p-1}<N<2+\dfrac{4(p+1)}{p-1}\left (\sqrt{\dfrac{2p}{p+1}}+\sqrt{\dfrac{2p}{p+1}-\sqrt{\dfrac{2p}{p+1}}}\right )$.
\end{itemize}
\vskip .15in

Let us recall that if we take
\begin{equation*}
\Gamma=\dfrac{4+\alpha+\beta}{p-1}\left (\dfrac{4+\alpha+\beta p}{p-1}+2 \right )\left (N-2-\dfrac{4+\alpha+\beta}{p-1} \right )
\left (N-4-\dfrac{4+\alpha+\beta p}{p-1} \right ),
\end{equation*}then
\begin{equation*}
u_{\Gamma}(r)=\Gamma^{\frac{1}{p-1}}r^{-\frac{4+\alpha+\beta}{p-1}}
\end{equation*}is a singular solution of (\ref{eq:1.1}) in $\mathbb{R}^N\backslash \{0\}$.
By the well-known weighted Hardy-Rellich inequality (\cite{Ghoussoub}) with the best constant
\begin{equation*}
\int_{\mathbb{R}^N} \dfrac{|\Delta \psi|^2}{|x|^{\beta}}dx \ge \frac{(N+\beta)^2(N-4-\beta)^2}{16}
\int_{\mathbb{R}^N} \dfrac{\psi^2}{|x|^{4+\beta}} dx,\quad \forall \psi \in H^2_{loc}(\mathbb{R}^N),
\end{equation*}we conclude that the singular solution $u_{\Gamma}$ is stable in $\mathbb{R}^N\backslash \{0\}$ if and only if
\begin{equation*}
\mathfrak{f}(N)=p\Gamma \le \frac{(N+\beta)^2(N-4-\beta)^2}{16}=\mathfrak{F}(N).
\end{equation*}Here $-1-\dfrac{\sqrt{1+(N-1)^2}}{2} \le \beta \le \dfrac{N-4}{2}$.\vskip .15in

\noindent {\bf Proof of Theorem \ref{eq:t1.1}.}
Since $u \in W^{2,2}(B_2 \backslash B_1)$, $|x|^{\alpha}|u|^{p+1} \in L_{loc}^1(\mathbb{R}^N
\backslash \{0\})$ and $|x|^{-\beta}|\Delta u|^2 \in L_{loc}^1(\mathbb{R}^N\backslash \{0\})$, we can assume that there exists a $\Psi \in W^{2,2}(\mathbb{S}^{N-1}) \cap L^{p+1}(\mathbb{S}^{N-1})$, such that
in polar coordinates
\begin{equation*}
u(r,\theta)=r^{-\frac{4+\alpha+\beta}{p-1}}\Psi (\theta).
\end{equation*}Substituting into (\ref{eq:1.1}) to get
\begin{equation*}
\Delta_{\theta}^2 \Psi-\Upsilon\Delta_{\theta} \Psi+\Gamma \Psi=|\Psi|^{p-1}\Psi,
\end{equation*}where
\begin{eqnarray*}
& \Upsilon =\lambda
(N-2-\lambda)+ \left (\dfrac{4+\alpha+\beta p}{p-1}+2 \right )\left (N-4-\dfrac{4+\alpha+\beta p}{p-1}\right ),&\\[0.15cm]
& \Gamma =\lambda (N-2-\lambda)\left (\dfrac{4+\alpha+\beta p}{p-1}+2 \right )
\left (N-4-\dfrac{4+\alpha+\beta p}{p-1} \right ).&
\end{eqnarray*}Multiplying the above equation by $\Psi$ and integration by parts yields
\begin{equation}\label{eq:3.4}
\int_{\mathbb{S}^{N-1}} |\Delta_{\theta} \Psi|^2+\Upsilon |\nabla_{\theta}\Psi|^2 +\Gamma \Psi^2=\int_{\mathbb{S}^{N-1}}|\Psi|^{p+1}.
\end{equation}

Since $u$ is a stable solution, we can take a test function $r^{-\frac{N-4-\beta}{2}}\Psi(\theta)\xi_{\varepsilon}(r)$ and obtain
\begin{equation}\label{eq:3.5}
p\int_{\mathbb{R}^N} |x|^{\alpha}|u|^{p-1}\left (r^{-\frac{N-4-\beta}{2}} \Psi(\theta)\xi_{\varepsilon}(r) \right )^2 dx
\le \int_{\mathbb{R}^N} \dfrac{\left |\Delta \left (r^{-\frac{N-4-\beta}{2}} \Psi(\theta)\xi_{\varepsilon}(r) \right )\right |^2}{|x|^{\beta}}dx.
\end{equation}Here, for any $\varepsilon >0$, we choose $\xi_{\varepsilon} \in C_0^2\left (\left (\frac{\varepsilon}{3},\frac{3}{\varepsilon}\right ) \right )$ such that
$\xi_{\varepsilon} \equiv 1$ in $\left (\varepsilon, \frac{1}{\varepsilon}\right )$ and
\begin{equation*}
r|\xi_{\varepsilon}'(r)|+r^2|\xi_{\varepsilon}''(r)| \le C,
\end{equation*}for all $r>0$. Then one can easily deduce that
\begin{equation*}
\int_0^{\infty}r^{-1} \xi_{\varepsilon}^2(r)dr \ge \int_{\varepsilon}^{\frac{1}{\varepsilon}}r^{-1}dr =2|\ln \varepsilon|,
\end{equation*}and
\begin{equation*}
\int_0^{\infty}\Big [r|\xi_{\varepsilon}'(r)|^2+r^3|\xi_{\varepsilon}''(r)|^2+|\xi_{\varepsilon}'(r)\xi_{\varepsilon}(r)|+r|\xi_{\varepsilon}(r)
\xi_{\varepsilon}''(r)| \Big ]dr \le C.
\end{equation*}Applying the coordinate transformation to the left hand side of (\ref{eq:3.5}), we get
\begin{align}\label{eq:3.6}
p\int_0^{+\infty} \int_{\mathbb{S}^{N-1}} r^{\alpha}|u|^{p-1}\left (r^{-\frac{N-4-\beta}{2}} \Psi(\theta)\xi_{\varepsilon}(r) \right )^2r^{N-1}dr d\theta \nonumber \\[0.1cm]
=p \left (\int_{\mathbb{S}^{N-1}} |\Psi|^{p+1} d \theta \right )\left (\int_0^{+\infty} r^{-1}\xi_{\varepsilon}^2(r) dr\right ).
\end{align}A direct calculation finds
\begin{align*}
\Delta \left (r^{-\frac{N-4-\beta}{2}} \Psi(\theta)\xi_{\varepsilon}(r) \right )= & -\dfrac{(N+\beta)(N-4-\beta)}{4}r^{-\frac{N-\beta}{2}}\xi_{\varepsilon}(r)\Psi(\theta)
+r^{-\frac{N-\beta}{2}}\xi_{\varepsilon}(r)\Delta_{\theta} \Psi \\[0.1cm]
& +(3+\beta)r^{-\frac{N-2-\beta}{2}}\xi'_{\varepsilon}(r)\Psi(\theta)+r^{-\frac{N-4-\beta}{2}}\xi''_{\varepsilon}(r)\Psi(\theta),
\end{align*}and inserting into the right hand side of (\ref{eq:3.5}) yields
\begin{align}\label{eq:3.7}
\int_{\mathbb{R}^N}& |x|^{-\beta} \left |\Delta \left (r^{-\frac{N-4-\beta}{2}} \Psi(\theta)\xi_{\varepsilon}(r) \right )\right |^2 dx \nonumber \\[0.13cm]
\le & \left [\int_{\mathbb{S}^{N-1}} \left (|\Delta_{\theta} \Psi|^2+\dfrac{(N+\beta)(N-4-\beta)}{2}|\nabla_{\theta} \Psi|^2+\dfrac{(N+\beta)^2(N-4-\beta)^2}{16}\Psi^2 \right ) d\theta \right ]\nonumber \\[0.13cm]
& \times \left (\int_0^{+\infty} r^{-1}\xi_{\varepsilon}^2(r) dr \right ) \nonumber \\[0.13cm]
& +O \left \{\int_0^{+\infty} \left [r|\xi_{\varepsilon}'(r)|^2+r^3|\xi_{\varepsilon}''(r)|^2+|\xi_{\varepsilon}'(r)|\xi_{\varepsilon}(r)+r\xi_{\varepsilon}(r)|\xi_{\varepsilon}''(r)| \right ] dr \right \}\nonumber \\[0.13cm]
& \times \int_{\mathbb{S}^{N-1}} \left [\Psi(\theta)^2+|\nabla_{\theta}\Psi(\theta)|^2 \right ] d\theta.
\end{align}Put (\ref{eq:3.6}) and (\ref{eq:3.7}) back into (\ref{eq:3.5}), take $\varepsilon \to 0$,
and pass to the limit to obtain
\begin{align*}
& p \int_{\mathbb{S}^{N-1}} |\Psi|^{p+1} d\theta \\[0.1cm]
&\le
\int_{\mathbb{S}^{N-1}} \left [|\Delta_{\theta} \Psi|^2+\dfrac{(N+\beta)(N-4-\beta)}{2}|\nabla_{\theta} \Psi|^2+\dfrac{(N+\beta)^2(N-4-\beta)^2}{16}\Psi^2 \right ] d\theta.
\end{align*}\vskip .05in

Now, combining the above inequality with (\ref{eq:3.4}), we have
\begin{align*}
\int_{\mathbb{S}^{N-1}} (p-1) & |\Delta_{\theta} \Psi|^2+\left (p\Upsilon-\dfrac{(N+\beta)(N-4-\beta)}{2} \right )|\nabla_{\theta} \Psi |^2 \\[0.1cm]
&+\left (p\Gamma-\dfrac{(N+\beta)^2(N-4-\beta)^2}{16}\right )\Psi^2 \le 0.
\end{align*}Since $4+\beta+\dfrac{8+2\alpha+2\beta}{p-1}<N<N_{\alpha,\beta}(p)$, it implies from the definition of $N_{\alpha,\beta}(p)$, (\ref{eq:3.2}) and (\ref{eq:3.3}) that
\begin{equation*}
\Psi(\theta)\equiv 0.
\end{equation*}Therefore, we get
\begin{equation*}
u \equiv 0. \eqno \square
\end{equation*}\vskip .2in

\section{Proof of Theorem \ref{eq:t1.2}}

\vskip .1in

\noindent {\bf Proof of Theorem \ref{eq:t1.2}.} We divide the proof into three cases.\vskip 0.1in

\noindent {\bf Case I.} $5 \le N < 4+\beta+\dfrac{8+2\alpha+2\beta}{p-1}$.\vskip 0.1in

Since $N<4+\beta+\dfrac{8+2\alpha+2\beta}{p-1}$, it implies from (\ref{eq:2.4}) that as $R \to +\infty$,
\begin{equation*}
\int_{B_R(x)} \left [\dfrac{|\Delta u|^2}{|z|^{\beta}} +|z|^{\alpha} |u|^{p+1} \right ]dz \le C R^{N-4-\beta-\frac{8+2\alpha+2\beta}{p-1}} \to
0.
\end{equation*}Therefore, we get
\begin{equation*}
u \equiv 0.
\end{equation*}

\noindent {\bf Case II.} $N=4+\beta+\dfrac{8+2\alpha+2\beta}{p-1}$.\vskip 0.1in

From the inequality (\ref{eq:2.4}), we obtain that
\begin{equation*}
\int_{\mathbb{R}^N} \left [\dfrac{|\Delta u|^2}{|z|^{\beta}} +|z|^{\alpha} |u|^{p+1} \right ]dz <+\infty,
\end{equation*}implies
\begin{equation*}
\lim\limits_{R\to +\infty} \int_{\mathfrak{D}} \left [\dfrac{|\Delta u|^2}{|z|^{\beta}}+|z|^{\alpha}
|u|^{p+1} \right ]dz=0,
\end{equation*}where $\mathfrak{D}:=B_{2R}(x) \backslash B_R(x)$. Applying (\ref{eq:2.3}) and H\"{o}lder's inequality yields
\begin{align*}
\int_{B_R(x)} & \left [\dfrac{|\Delta u|^2}{|z|^{\beta}}+|z|^{\alpha} |u|^{p+1} \right ] dz
\le  CR^{-2} \int_{\mathfrak{D}}\dfrac{|u\Delta u|}{|z|^{\beta}} dz +CR^{-4}\int_{\mathfrak{D}}\dfrac{u^2}{|z|^{\beta}} dz\\[0.1cm]
& \le C\mathfrak{C}R^{-2}\left (\int_{\mathfrak{D}} |z|^{\alpha}|u|^{p+1} dz\right )^{\frac{1}{p+1}}
\left (\int_{\mathfrak{D}} |z|^{-\frac{2\alpha+\beta (p+1)}{p-1}} dz\right )^{\frac{p-1}{2(p+1)}} \\[0.1cm]
&\;\;\ +CR^{-4} \left (\int_{\mathfrak{D}}|z|^{\alpha}|u|^{p+1} dz\right )^{\frac{2}{p+1}}\left (
\int_{\mathfrak{D}} |z|^{-\frac{2\alpha+\beta (p+1)}{p-1}} dz \right )^{\frac{p-1}{p+1}} \\[0.1cm]
& \le C\mathfrak{C}R^{\left [N-4-\beta-\frac{8+2\alpha+2\beta}{p-1} \right ]\frac{p-1}{2(p+1)}} \left (\int_{\mathfrak{D}}
|z|^{\alpha}|u|^{p+1} dz\right )^{\frac{p-1}{2(p+1)}} \\[0.1cm]
& \;\;\ +CR^{\left [N-4-\beta-\frac{8+2\alpha+2\beta}{p-1} \right ]\frac{p-1}{p+1}}\left (
\int_{\mathfrak{D}} |z|^{\alpha}|u|^{p+1} dz\right )^{\frac{p-1}{p+1}},
\end{align*}where $\mathfrak{C}=\left (\displaystyle \int_{\mathfrak{D}} \dfrac{|\Delta u|^2}{|z|^{\beta}} dz \right )^{\frac{1}{2}}$.
From $N=4+\beta+\dfrac{8+2\alpha+2\beta}{p-1}$, it implies that the right hand side of the above inequality converges to $0$ as $R \to +\infty$. Therefore, we obtain
\begin{equation*}
u \equiv 0.
\end{equation*}

\noindent {\bf Case III.} $4+\beta+\dfrac{8+2\alpha+2\beta}{p-1}<N<N_{\alpha,\beta}(p)$.\vskip 0.1in

First, we will obtain some properties of the
function $\mathcal{M}$.

\begin{lemma}\label{eq:l4.1}
$\lim\limits_{r \to +\infty} \mathcal{M}(r;0,u) <+\infty$.
\end{lemma}

\begin{proof}
The proof mainly use the estimate (\ref{eq:2.4}) and the monotonicity of the function $\mathcal{M}(r;0,u)$ in $r$.

Applying (\ref{eq:2.4}) to estimate the first term in the right hand side of (\ref{eq:2.14}) yields
\begin{align*}
&  r^{\frac{8+2\alpha+2\beta}{p-1}+4+\beta-N} \int_{B_r} \left [ \dfrac{1}{2}
\dfrac{(\Delta u)^2}{|x|^{\beta}}-\dfrac{1}{p+1}|x|^{\alpha} |u|^{p+1} \right ]dx \\[0.1cm]
& \le C r^{\frac{8+2\alpha+2\beta}{p-1}+4+\beta-N} r^{N-4-\beta-\frac{8+2\alpha+2\beta}{p-1}} \\
& \le C.
\end{align*}Utilize H\"{o}lder's inequality to estimate the second term in the right hand side of (\ref{eq:2.14})
\begin{align*}
& r^{\frac{8+2\alpha+2\beta}{p-1}+1-N} \int_{\partial B_r} u^2 \le \dfrac{1}{r} \int_r^{2r} \left (\mu^{\frac{8+2\alpha+2\beta}{p-1}+1-N} \int_{\partial B_{\mu}} u^2 dS \right )d \mu  \\[0.1cm]
& \le \dfrac{1}{r} \left ( \int_{B_{2r}\backslash B_r} \left ( |x|^{\frac{8+2\alpha+2\beta}{p-1}+1-N-\frac{2\alpha}{p+1}} \right
)^{\frac{p+1}{p-1}} \right )^{\frac{p-1}{p+1}} \left ( \int_{B_{3r}} |x|^{\alpha} |u|^{p+1} \right )^{\frac{2}{p+1}} \\[0.1cm]
& \le C r^{\left [\frac{8+2\alpha+2\beta}{p-1}-N-\frac{2\alpha}{p+1}+N\frac{p-1}{p+1}\right ]} r^{\frac{2}{p+1}\left
[N-4-\beta-\frac{8+2\alpha+2\beta}{p-1} \right ]}\\[0.1cm]
& \le C.
\end{align*}Similarly, we find
\begin{align*}
\dfrac{d}{dr} \left (r^{2\lambda+2-N}
\int_{\partial B_r} u^2 \right ) & \le \dfrac{1}{r^2} \int^{2r}_r \int_{\iota}^{\iota+r} \dfrac{d}{d \mu} \left (\mu^{2\lambda+2-N}
\int_{\partial B_{\mu}} u^2 \right )d \mu d\iota \\[0.1cm]
& \le C.
\end{align*}By the interpolation inequality and H\"{o}lder's inequality, we get
\begin{align}\label{eq:4.1}
\int_{B_r} |\nabla u|^2 \le & Cr^2 \int_{B_r} |\Delta u|^2 +Cr^{-2} \int_{B_r} u^2 \nonumber \\[0.1cm]
\le & Cr^2 \int_{B_r} |x|^{\beta}\dfrac{|\Delta u|^2}{|x|^{\beta}} +Cr^{-2} \left (
\int_{B_r} |x|^{\alpha}|u|^{p+1} \right )^{\frac{2}{p+1}} \left (\int_{B_r} |x|^{-\frac{2\alpha}{p-1}} dx \right )^{\frac{p-1}{p+1}} \nonumber \\[0.1cm]
\le & Cr^{N-2-\frac{8+2\alpha+2\beta}{p-1}}.
\end{align}Then, it implies that
\begin{align*}
r^{\frac{8+2\alpha+2\beta}{p-1}+3-N} \int_{\partial B_r}
|\nabla u|^2 dS \le \dfrac{1}{r} \int_r^{2r} \left (\mu^{\frac{8+2\alpha+2\beta}{p-1}+3-N} \int_{\partial B_{\mu}} |\nabla u|^2 dS \right )d \mu \le C.
\end{align*}Therefore, we get the boundedness of the fifth and sixth terms in the right hand side of (\ref{eq:2.14}). Utilizing
H\"{o}lder's inequality and (\ref{eq:4.1}), we find
\begin{align*}
\dfrac{1}{r^2} & \int_r^{2r} \int_{\iota}^{\iota+r} \dfrac{\mu^3}{2} \dfrac{d}{d \mu} \left [
\mu^{2\lambda +1-N} \int_{\partial B_{\mu}} \Big (\lambda \mu^{-1}u+\dfrac{\partial u}{\partial r}
\Big )^2 \right ] d \mu d\iota \\[0.13cm]
= & \dfrac{1}{2r^2}\int_r^{2r} \left \{(\iota+r)^{2\lambda +4-N} \int_{\partial B_{\iota+r}}
\Big [\lambda (\iota+r)^{-1}u+\dfrac{\partial u}{\partial r} \Big ]^2-
\iota^{2\lambda+4-N} \int_{\partial B_r}
\Big [\lambda \iota^{-1}u+\dfrac{\partial u}{\partial r} \Big ]^2\right \} \\[0.13cm]
& -\dfrac{3}{2r^2} \int_r^{2r} \int_{\iota}^{\iota+r} \mu^{2\lambda+3-N} \int_{\partial B_{\mu}}\left (\lambda
\mu^{-1} u+\dfrac{\partial u}{\partial r} \right )^2  \\[0.13cm]
\le & \dfrac{C}{r^2} \int_{B_{3r}\backslash B_r} |x|^{2\lambda +2-N} \left (u^2 +
|x|^2 \left (\dfrac{\partial u}{\partial r} \right )^2 \right )dx \\[0.1cm]
\le & C.
\end{align*}Consequently, we obtain the desired result.
\end{proof}

\begin{lemma}\label{eq:l4.2} For all $\kappa>0$, define
{\it blowing down} sequences
\begin{equation*}
u^{\kappa}(x):=\kappa^{\frac{4+\alpha+\beta}{p-1}}u(\kappa x),
\end{equation*}then $u^{\kappa}$ strongly converges to $u^{\infty}$ in $W_{loc}^{1,2}
(\mathbb{R}^N) \cap L_{loc}^{p+1}(\mathbb{R}^N)$. Furthermore, $u^{\infty}$ is a homogeneous stable solution of (\ref{eq:1.1}). 
\end{lemma}

\begin{proof}
Since $u$ is a stable solution of (\ref{eq:1.1}), we can find
\begin{align}\label{eq:4.2}
& p \int_{\mathbb{R}^N} |x|^{\alpha}|u^{\kappa}|^{p-1} \zeta^2(x)dx=p\int_{\mathbb{R}^N}|\kappa x|^{\alpha}\kappa^{4+\beta}|u(\kappa x)|^{p-1}\zeta^2(x)dx \nonumber \\[0.1cm]
& =p\kappa^{4+\beta-N}\int_{\mathbb{R}^N}|y|^{\alpha}|u(y)|^{p-1}\psi^2(y)dy\quad\; \mbox{taking}\;\ \psi(y):=\zeta(x),\; x=\dfrac{y}{\kappa} \nonumber \\[0.1cm]
& \le \kappa^{4+\beta-N}\int_{\mathbb{R}^N}\dfrac{|\Delta \psi(y)|^2}{|y|^{\beta}}dy \nonumber \\[0.1cm]
& =\int_{\mathbb{R}^N} \dfrac{|\Delta \zeta|^2}{|x|^{\beta}}dx.
\end{align}Thus, $u^{\kappa}$ is a stable solution of (\ref{eq:1.1}). Furthermore, from (\ref{eq:2.4}),
it implies that
\begin{align*}
& \int_{B_r(x)} \left [|y|^{-\beta} \left (\Delta u^{\kappa}\right )^2+|y|^{\alpha} |u^{\kappa}|^{p+1} \right ] dy \nonumber \\[0.1cm]
& = \kappa^{4+\beta+\frac{8+2\alpha+2\beta}{p-1}-N} \int_{B_{\kappa r}(x)} \left [ |z|^{-\beta}|\Delta(z)|^2+|z|^{\alpha} |u(z)|^{p+1} \right ] dz \nonumber \\[0.1cm]
& \le Cr^{N-4-\beta-\frac{8+2\alpha+2\beta}{p-1}},
\end{align*}and applying H\"{o}lder's inequality yields
\begin{align*}
\int_{B_r(x)} \left | u^{\kappa} \right |^2 dz & \le \left (\int_{B_r(x)} |z|^{\alpha} \left |u^{\kappa}\right |^{p+1} dz\right )^{\frac{2}{p+1}}
\left (\int_{B_r(x)} |z|^{-\frac{2\alpha }{p-1}} dz \right )^{\frac{p-1}{p+1}}\\[0.1cm]
& \le  Cr^{N-2\lambda}.
\end{align*}Clearly, we also obtain
\begin{align*}
\int_{B_r(x)} |\Delta u^{\kappa}|^2 dz & =\int_{B_{\kappa r}(x)} \kappa^{2\lambda+4-N}|z|^{\beta}\dfrac{|\Delta u(z)|^2}{|z|^{\beta}}dz\\
& \le Cr^{N-4-2\lambda}.
\end{align*}\vskip .06in

By the application of the elliptic regularity theory,
it implies that $u^{\kappa}$ are  uniformly bounded in $W^{2,2}_{loc}(\mathbb{R}^N)$.
Again $u\in C^4(\mathbb{R}^N)$ implies $u^{\kappa} \in L^{p+1}_{loc}(\mathbb{R}^N)$. Then
we can suppose that $u^{\kappa} \rightharpoonup u^{\infty}$ weakly
in $W^{2,2}_{loc}(\mathbb{R}^N) \cap L^{p+1}_{loc}(\mathbb{R}^N)$ (if necessary, we can extract a subsequence).
Now, using the standard embeddings, we get
$u^{\kappa} \to u^{\infty}$ strongly in $W^{1,2}_{loc}(\mathbb{R}^N)$. Therefore, applying the interpolation inequality
between $L^q$ spaces with $q\in (1,p+1)$, we get that, for any ball $B_r$
\begin{equation}\label{eq:4.3}
\|u^{\kappa}-u^{\infty}\|_{L^q(B_r)}\le \|u^{\kappa}-u^{\infty}\|_{L^1(B_r)}^t \|u^{\kappa}-u^{\infty}\|_{L^{p+1}(B_r)}^{1-t} \to 0,\;\;\mbox{as}\;\; \kappa \to +\infty,
\end{equation}where $t \in (0,1)$ satisfying $\dfrac{1}{q}=t +\dfrac{1-t}{p+1}$. Next, combining with the definition of $u^{\kappa}$ and (\ref{eq:4.2}), we conclude that, for any $\zeta \in C_0^2 (\mathbb{R}^N)$
\begin{align*}
\int_{\mathbb{R}^N} \dfrac{\Delta u^{\infty}}{|x|^{\beta}}\Delta \zeta -|x|^{\alpha} |u^{\infty}|^{p-1}u^{\infty} \zeta = \lim\limits_{\kappa \to \infty}
\int_{\mathbb{R}^N} \dfrac{\Delta u^{\kappa}}{|x|^{\beta}}\Delta \zeta -|x|^{\alpha} |u^{\kappa}|^{p-1}u^{\kappa} \zeta,  \\[0.15cm]
\int_{\mathbb{R}^N} \dfrac{\left (\Delta \zeta \right )^2}{|x|^{\beta}}-p |x|^{\alpha} |u^{\infty}|^{p-1} \zeta^2 = \lim\limits_{\kappa \to \infty} \int_{\mathbb{R}^N} \dfrac{(\Delta \zeta)^2}{|x|^{\beta}} -p |x|^{\alpha} |u^{\kappa}|^{p-1} \zeta^2  \ge 0,
\end{align*}that is, $u^{\infty} \in W_{loc}^{2,2}
(\mathbb{R}^N) \cap L_{loc}^{p+1}(\mathbb{R}^N)$ is a stable solution of (\ref{eq:1.1}) in $\mathbb{R}^N$. \vskip .1in

From the boundedness and monotonicity of $\mathcal{M}(r;0,u)$, it implies that for any $0<r_1<r_2<+\infty$,
\begin{equation*}
\lim\limits_{\kappa \to \infty} \Big [\mathcal{M}(\kappa r_2;0,u)-\mathcal{M}(\kappa r_1; 0,u)\Big ]=0.
\end{equation*}Again using the scaling invariance and Theorem \ref{eq:t2.1}, we get
\begin{align*}
0 & =\lim\limits_{\kappa \to \infty} \left [\mathcal{M}\left (r_2;0,u^{\kappa}\right )-\mathcal{M}\left (r_1;0,u^{\kappa}
\right )\right ] \\[0.1cm]
& = \lim\limits_{\kappa \to \infty} \int_{r_1}^{r_2} \dfrac{d }{d \mu}\mathcal{M}\left (\mu; 0,u^{\kappa} \right ) d \mu \\[0.1cm]
& \ge C(N,p,\alpha,\beta) \int_{B_{r_2} \backslash B_{r_1}} |x|^{2+2\lambda-N}\left (\lambda \mu^{-1}
u^{\infty}+\dfrac{\partial u^{\infty}}{\partial \mu}\right )^2dx.
\end{align*}Adopting the same calculation as Theorem \ref{eq:t2.1}, we obtain
that $u^{\infty}$ is homogeneous.
\end{proof}

\begin{lemma}\label{eq:l4.3}
$\lim\limits_{r \to \infty} \mathcal{M} (r;0,u)=0$.
\end{lemma}

\noindent {\it Proof.} Since $u^{\infty}$ is a homogeneous, stable solution of (\ref{eq:1.1}),
it implies from Theorem \ref{eq:t1.1} that
\begin{equation*}
u^{\infty} \equiv 0.
\end{equation*}Combining (\ref{eq:4.3}) with the above equality, we find that
\begin{equation*}
\lim\limits_{\kappa \to +\infty} u^{\kappa} =0, \;\ \mbox{stongly}\; \mbox{in}\; \ L^2(B_6),
\end{equation*}
i.e.,
\begin{equation*}
\lim\limits_{\kappa \to +\infty} \int_{B_6} |u^{\kappa}|^2 =0.
\end{equation*}From the uniform boundedness of $\Delta u^{\kappa}$ in $L^2(B_6)$, we get
\begin{equation*}
\lim\limits_{\kappa \to \infty} \int_{B_6} \left |u^{\kappa}\Delta u^{\kappa}\right | \le \lim\limits_{\kappa \to \infty}
\left (\int_{B_6} \left |u^{\kappa}\right |^2 \right )^{\frac{1}{2}} \left ( \int_{B_6}\left |\Delta u^{\kappa}\right |^2 \right )^{\frac{1}{2}}=0.
\end{equation*}Therefore, it implies from (\ref{eq:2.3}) that
\begin{equation*}
\lim\limits_{\kappa \to +\infty} \int_{B_1} \left | \Delta u^{\kappa}\right |^2 +|x|^{\alpha} |u^{\kappa}|^{p+1} \le
C \lim\limits_{\kappa \to +\infty} \int_{B_6} \left |u^{\kappa}\right |^2+\left |u^{\kappa} \Delta u^{\kappa}\right | =0.
\end{equation*}A direct application of the interior $L^p$-estimates gets
\begin{equation*}
\lim\limits_{\kappa \to +\infty} \int_{B_2} \sum\limits_{j \le 2}
|\nabla^j u^{\kappa}|=0,
\end{equation*}implies
\begin{equation*}
\int_1^2 \left ( \sum\limits_{i=1}^{\infty} \int_{\partial B_r} \sum\limits_{j\le 2}
|\nabla^j u^{\kappa_i}|^2\right ) dr \le \sum\limits_{i=1}^{\infty} \int_{B_{2r}\backslash B_r}
\sum\limits_{j \le 2} |\nabla^j u^{\kappa_i}|^2 \le 1.
\end{equation*}Then, let us note that there exists a $\gamma \in (1,2)$ such that
\begin{equation*}
\lim\limits_{\kappa \to \infty} \|u^{\kappa}\|_{W^{2,2}(\partial B_{\gamma})}=0.
\end{equation*}

Now, combing the above results with the scaling invariance of $\mathcal{M}(r;0,u)$, we obtain
\begin{equation*}
\lim\limits_{i \to \infty}\mathcal{M}(\kappa_i \gamma; 0,u)=\lim\limits_{i \to \infty}\mathcal{M}(\gamma;0,u^{\kappa_i})=0.
\end{equation*}Again since $\kappa_i \gamma \to +\infty$ and $\mathcal{M}(r;0,u)$ is non-decreasing in $r$, we get
\begin{equation*}
\lim\limits_{r \to \infty} \mathcal{M}(r;0,u)=0.\eqno \square
\end{equation*}

Since $u \in C^4(\mathbb{R}^N)$, we get $\lim\limits_{r \to 0} \mathcal{M}(r;0,u)=0$. Again using the monotonicity of $\mathcal{M}(r;0,u)$
and Lemma \ref{eq:l4.3}, we get
\begin{equation*}
\mathcal{M}(r;0,u)=0,\quad \mbox{for}\;\; \mbox{all}\;\; r>0.
\end{equation*}Therefore, combining with Theorem \ref{eq:t2.1}, we conclude that $u$ is homogeneous and by Theorem \ref{eq:t1.1}
\begin{equation*}
u\equiv 0. \eqno \square
\end{equation*}
\vskip .3in

\noindent {\bf Acknowledge:} The author wishes to express his warmest thanks to Chern Institute of Mathematics, Nankai University (where part of this work was done) for their warm hospitality.

\end{document}